\documentclass{article}
\usepackage{amsmath}
\usepackage{amsthm}
\usepackage{amssymb}
\usepackage[utf8]{inputenc}
\usepackage{natbib}
\usepackage{graphicx}
\usepackage{subfig}
\usepackage{algorithm}
\usepackage{algorithmic}
\usepackage{cases}
\usepackage{hyperref}
\usepackage{cleveref}
\usepackage{geometry}
\usepackage{fullpage}
\usepackage{authblk}
\usepackage{xcolor}

\newtheorem{prop}{Proposition}

\newtheorem{lem}{Lemma}
\newtheorem{thm}{Theorem}
\newtheorem{asmp}{Assumption}
\newtheorem{coro}{Corollary}
\newtheorem{exm}{Example}
\newtheorem{rmk}{Remark}

\DeclareMathOperator*{\argmin}{arg\,min}

\title{Well-Posedness and Efficient Algorithms for Inverse Optimal Transport with Bregman Regularization}
\author[1]{Chenglong Bao \thanks{Email: \texttt{clbao@tsinghua.edu.cn}}}
\author[2]{Zanyu Li \thanks{Corresponding author. Email: \texttt{lizy21@mails.tsinghua.edu.cn}}}
\author[3]{Yunan Yang \thanks{Email: \texttt{yunan.yang@cornell.edu}}}
\affil[1]{Yau Mathematical Sciences Center, Tsinghua University, Beijing 100084, China}
\affil[2]{Department of Mathematical Sciences, Tsinghua University, Beijing 100084, China}
\affil[3]{Department of Mathematics, Cornell University, Ithaca, New York 14853, USA}
\date{}

\begin{document}

\maketitle

\begin{abstract}
    This work analyzes the inverse optimal transport (IOT) problem under Bregman regularization. We establish well-posedness results, including existence, uniqueness (up to equivalence classes of solutions), and stability, under several structural assumptions on the cost matrix. On the computational side, we investigate the existence of solutions to the optimization problem with general constraints on the cost matrix and provide a sufficient condition guaranteeing existence. In addition, we propose an inexact block coordinate descent (IBCD) method for the problem with a strongly convex penalty term. In particular, when the penalty is quadratic, the subproblems admit a diagonal Hessian structure, which enables highly efficient element-wise Newton updates. We establish a linear convergence rate for the algorithm and demonstrate its practical performance through numerical experiments, including the validation of stability bounds, the investigation of regularization effects, and the application to a marriage matching dataset.
\end{abstract}

\section{Introduction}

Optimal Transport (OT) \cite{villani2008optimal} has become a fundamental tool in modern optimization, offering a powerful framework to analyze structured relationships between probability distributions. In recent years, it has attracted significant attention and has found widespread application in various areas, including data science \cite{peyre2019computational}, domain adaptation \cite{courty2016optimal}, and signal processing \cite{kolouri2017optimal}. 

The classic discrete optimal transport problem can be formulated as follows. Given two discrete probability distributions $\mu,\nu$, namely, $\mu,\nu\in\{x\in\mathbb{R}^n: \sum_{i=1}^n x_i = 1, x_i\geq0, 1\leq i\leq n\}$, define the set of transport plans: $\mathcal{U}(\mu,\nu) := \{ X\in\mathbb{R}^{n\times n} : X\mathbf{1} = \mu, X^\top\mathbf{1} = \nu, X\in\mathbb{R}_+^{n\times n} \}$. Given a cost matrix $C\in\mathbb{R}_+^{n\times n}$, the classic optimal transport problem is the following linear programming:
\begin{equation}\label{eqn:ot}
       \min_{X\in\mathcal{U}(\mu,\nu)}  \langle C, X\rangle.
\end{equation}
Due to the high computational cost of solving the standard linear programming formulation of optimal transport, one typically considers the following \textit{regularized optimal transport} problem instead:
\begin{equation}\label{eqn:breg_ot}
    \min_{X \in \mathcal{U}(\mu,\nu)} \langle C, X \rangle + \gamma\, \phi(X),
\end{equation}
where $\gamma > 0$ is a regularization parameter and $\phi$ is a strictly convex, twice continuously differentiable function. A popular choice for $\phi$ is the \textit{entropy function},
\[
H(X) = \sum_{i,j=1}^n X_{ij} \log X_{ij} - X_{ij} + 1,
\]
mainly because it admits an efficient and easy-to-implement solution via the Sinkhorn algorithm~\cite{cuturi2013sinkhorn}. Entropy-regularized optimal transport has found widespread use in machine learning applications~\cite{genevay2019entropy,cuturi2013sinkhorn}. Other options of regularizers have also been explored, such as the quadratic regularizer $\frac{1}{2}\|X\|_F^2$~\cite{lorenz2021quadratically,essid2018quadratically}.

While classical OT seeks an optimal coupling given a cost function, \textit{Inverse Optimal Transport} (IOT)~\cite{stuart2020inverse,li2019learning} inverts this paradigm by aiming to recover the cost matrix from observed transport plans. This inverse problem has important applications in areas such as marriage matching~\cite{dupuy2019estimating}, machine learning~\cite{wang2023self,carlier2023sista,shi2023understanding,persiianov2024inverse}, and legal case analysis~\cite{yu2022explainable}. However, it also poses unique theoretical and computational challenges. Given an observed transport plan $\hat{X} \in \mathcal{U}(\mu,\nu)$, the goal is to find a cost matrix $C \in \mathbb{R}_+^{n \times n}$ such that $\hat{X}$ solves the optimal transport problem~\eqref{eqn:ot} defined by $C$. In general, this inverse problem is ill-posed, as the solution to~\eqref{eqn:ot} may not be unique. To address this, we consider the \textit{inverse regularized optimal transport} problem. Define the forward map
\[
\mathcal{F}(C) := \arg\min_{X \in \mathcal{U}(\mu,\nu)} \left\{ \langle C, X \rangle + \gamma\, \phi(X) \right\}.
\]
The strict convexity of $\phi$ ensures that $\mathcal{F}$ is well-defined and single-valued. The inverse problem is then to recover a cost matrix $C$ such that $\mathcal{F}(C) = \hat{X}$. That is, find $C \in \mathcal{F}^{-1}(\hat{X})$ given an observed plan $\hat{X} \in \mathcal{U}(\mu,\nu)$.

A foundational contribution to IOT by~\cite{stuart2020inverse} introduced a Bayesian framework for inferring unknown cost functions from noisy observations. Building on this,~\cite{chiu2022discrete} proposed a discrete probabilistic inverse optimal transport model that leverages entropy-regularized OT to characterize the manifold of cross-ratio equivalent cost matrices and develop an MCMC sampler for posterior inference. An in-depth theoretical study of $\ell_1$-regularized inverse optimal transport was presented by~\cite{andrade2023sparsistency}. This work derived sufficient conditions for the robust recovery of the ground cost's sparsity pattern and established connections between iOT, graphical Lasso, and classical Lasso under varying entropic regularizations. In a related direction, \cite{liu2019learning} introduced OT-SI, an algorithm that learns an optimal transport cost function using subset correspondences (i.e., side information indicating that specific subsets of points across two datasets are related), thereby improving dataset alignment. From an optimization perspective, \cite{ma2020learning} reformulated IOT as an unconstrained convex optimization problem and developed efficient numerical methods applicable to both discrete and continuous settings. Finally, \cite{andrade2025learning} introduce a general methodology based on a new class of "sharpened Fenchel-Young losses" to solve inverse problems over probability measures, providing stability guarantees for applications in inverse unbalanced optimal transport (iUOT) and inverse gradient flows (iJKO).

While these contributions have significantly advanced the field of IOT, a complete understanding of the problem's foundational well-posedness properties remains elusive. Existing theoretical efforts have made noteworthy progress. For example, \cite{chiu2022discrete} characterized the equivalence class of cost matrices for the specific case of entropy regularization. The identifiability (i.e., uniqueness) of the cost function, sometimes inferred from multiple observed transport plans, has been explored in works such as \cite{gonzalez2024identifiability,gonzalez2024nonlinear,li2019learning}. Notably, \cite{li2019learning} establishes injectivity and provides stability bounds when the cost matrix is constrained to the space of distance matrices under entropy regularization. 

However, two critical gaps persist in the current literature. First, much of the existing analysis emphasizes the uniqueness or identifiability of solutions while often overlooking the more fundamental question of \emph{existence}: given an observed transport plan, is there a guaranteed cost matrix that could have produced it? Second, most theoretical studies on IOT are narrowly focused on entropy regularization. This singular emphasis on entropy is limiting, as the forward OT problem has shown that alternative regularizer, such as quadratic, Fermi-Dirac, or other entropy-like functions, can provide unique advantages for modeling diverse data structures and constraints. 

To develop a more versatile and robust IOT theory, it is crucial to account for this diversity of regularizers. Bregman divergences provide a general framework for addressing this need. Beyond covering the common entropy and quadratic regularizers, Bregman divergences provide a broad spectrum of alternatives (as summarized in Table \ref{tab:bregman_func}), each potentially suited to specific data characteristics or modeling requirements. This generality significantly expands the applicability of IOT theory. Therefore, advancing IOT within the Bregman divergence framework is essential for fully realizing its potential.

To bridge these gaps, this paper conducts a comprehensive study on inverse Bregman-regularized optimal transport. Our main contributions are as follows:
\begin{itemize}
    \item[1.] \textbf{Well-Posedness for General Bregman IOT:} We present a comprehensive analysis of well-posedness, establishing conditions for the existence, uniqueness (both of solution equivalence classes and, under additional structural assumptions, of the cost matrix itself), and stability of the recovered cost in IOT with general Bregman regularizers. This significantly extends prior work, which has largely focused on the entropy case.
    
    \item[2.] \textbf{Generalized Single-Level Formulation and Existence Theory:} While the reformulation of the bi-level IOT problem into a single-level convex program is known in the entropy setting~\cite{ma2020learning}, we extend this formulation to arbitrary Bregman regularizers. Importantly, we provide a rigorous analysis of the existence of solutions to the resulting single-level problem, laying the foundation for algorithmic development.
    
    \item[3.] \textbf{Efficient BCD Algorithm with Convergence Guarantees:} We introduce an efficient Block Coordinate Descent (BCD) algorithm (see~\Cref{alg:practical_bcd}) for solving the proposed IOT formulation. When paired with common quadratic penalties, the BCD subproblems admit fast, element-wise Newton updates due to their diagonal Hessians. We prove linear convergence of the BCD algorithm, providing theoretical justification for its effectiveness.
    
    \item[4.] \textbf{Numerical Validation:} We conduct extensive numerical experiments to validate our theoretical results, including the stability bounds in~\Cref{prop:stab}, and to assess the practical performance of the BCD algorithm. Our experiments include synthetic datasets and real-world applications, including the inference of marriage matching preferences.
\end{itemize}

The remainder of this paper is organized as follows. \Cref{sec:notation} introduces notations and preliminaries on Bregman functions and their properties. \Cref{sec:forward_ot} delves into Bregman-regularized OT, reviewing the forward problem. \Cref{sec:well_posed} provides a well-posedness analysis by investigating the properties of the forward mapping. In \Cref{sec:iot_model} we develop our framework for the inverse problem, provide the convex reformulation and detail the proposed BCD algorithm. \Cref{sec:experiment} presents the results of our numerical experiments. Finally, \Cref{sec:conclusion} concludes with a summary and potential directions for future work.

\section{Notations and Preliminaries}\label{sec:notation}

Throughout this paper, we denote by $\hat{X} \in \mathbb{R}^{n \times n}$ the observed transport plan. We use $\mathbb{R}_{++}$ to denote the set of strictly positive real numbers and $\mathbb{R}_+$ to denote the set of nonnegative real numbers. Denote $\mathcal{S}^n$ the set of all symmetric matrices in $\mathbb{R}^{n\times n}$. Denote $(x)_+ := \max\{0,x\}$. Given two vectors $a,b\in\mathbb{R}^n$, define a matrix $a\oplus b$ with $(a\oplus b)_{ij} = a_i+b_j$. Fix two discrete probability measures $\mu$ and $\nu$, and define the transport polytope
\[
\mathcal{U}(\mu,\nu) := \left\{ X \in \mathbb{R}_+^{n \times n} : X \mathbf{1} = \mu,\, X^\top \mathbf{1} = \nu \right\}.
\]
We always assume that $\hat{X} \in \mathcal{U}(\mu,\nu)$. Additionally, define the set of admissible symmetric cost matrices as
\[
S_h := \left\{ C \in \mathbb{R}_+^{n \times n} : C = C^\top,\ \operatorname{diag}(C) = 0 \right\}.
\]
Let $\mathcal{E}$ be a finite-dimensional Euclidean space. The \emph{effective domain} (or simply \emph{domain}) of a function $\phi : \mathcal{E} \to (-\infty, +\infty]$ is defined as
\[
\operatorname{dom} \phi := \left\{ x \in \mathcal{E} : \phi(x) < +\infty \right\}.
\]
A convex function $\phi$ is said to be \emph{proper} if there exists at least one $x \in \mathcal{E}$ such that $\phi(x) < +\infty$, and $\phi(x) > -\infty$ for all $x \in \mathcal{E}$. It is said to be \emph{closed} if all its lower level sets $\{ x \in \mathcal{E} : \phi(x) \leq \alpha \}$ are closed for all $\alpha \in \mathbb{R}$. A function $\phi$ is \emph{essentially smooth} if it is differentiable on the interior of its domain, $\operatorname{int}(\operatorname{dom} \phi) \neq \emptyset$, and for any sequence $(x_k)_{k \in \mathbb{N}} \subset \operatorname{int}(\operatorname{dom} \phi)$ converging to a boundary point $x \in \operatorname{bd}(\operatorname{dom} \phi)$, we have
\[
\lim_{k \to +\infty} \left\| \nabla \phi(x_k) \right\| = +\infty.
\]
A function $\phi$ is said to be of \emph{Legendre type} if it is proper, closed, strictly convex on $\operatorname{int}(\operatorname{dom} \phi)$, and essentially smooth. The \emph{Fenchel conjugate} $\phi^* : \mathcal{E} \to (-\infty, +\infty]$ of a function $\phi$ is defined as
\[
\phi^*(y) := \sup_{x \in \operatorname{int}(\operatorname{dom} \phi)} \left\{ \langle x, y \rangle - \phi(x) \right\}.
\]
The conjugate $\phi^*$ is always a closed convex function. Moreover, if $\phi$ is closed and convex, then $(\phi^*)^* = \phi$. A function $\phi$ is of Legendre type if and only if its Fenchel conjugate $\phi^*$ is also of Legendre type. In this case, the gradient map $\nabla \phi$ is a homeomorphism between $\operatorname{int}(\operatorname{dom} \phi)$ and $\operatorname{int}(\operatorname{dom} \phi^*)$, with inverse $(\nabla \phi)^{-1} = \nabla \phi^*$.

Let $\phi : I \rightarrow \mathbb{R}$ be a \emph{Bregman function}, i.e., a proper, closed, convex function that is differentiable on $\operatorname{int}(\operatorname{dom} \phi) \neq \emptyset$, where $I \subset \mathbb{R}$ is an interval. The \emph{Bregman divergence} generated by $\phi$ is defined by
\[
B_\phi(x \| y) := \phi(x) - \phi(y) - (x - y)\phi^\prime(y),
\]
for all $x \in \operatorname{dom} \phi$ and $y \in \operatorname{int}(\operatorname{dom} \phi)$. It holds that $B_\phi(x \| y) \geq 0$ for any such $x$ and $y$. Moreover, if $\phi$ is strictly convex on $\operatorname{int}(\operatorname{dom} \phi)$, then $B_\phi(x \| y) = 0$ if and only if $x = y$. Bregman divergences are always convex in the first argument and are invariant under the addition of affine terms to their generator $\phi$. Assume that $(0,1) \subset I$. Let $\mathcal{I} := I^{n \times n}$. For matrices $X, Y \in \mathcal{I}$, we extend the definition of Bregman divergence and the generator function entrywise:
\[
B_\phi(X \| Y) := \sum_{i=1}^n \sum_{j=1}^m B_{\phi}\left(X_{ij} \| Y_{ij}\right), \quad
\phi(X) := \sum_{i=1}^n \sum_{j=1}^m \phi\left(X_{ij}\right).
\]
We now state a basic assumption on the Bregman generator $\phi$:

\begin{asmp}\label{asmp:breg}
The effective domain of $\phi : I \rightarrow \mathbb{R}$ is an interval $I = \operatorname{dom} \phi$ with $(0,1) \subseteq \operatorname{int}(I)$. Moreover, $\phi$ is of Legendre type and is continuously differentiable ($C^1$) on $\operatorname{int}(I)$.
\end{asmp}

Under \Cref{asmp:breg}, we define the limiting derivatives
\begin{equation}\label{eq:limit_derivative}
\phi^\prime_0 := \lim_{x \rightarrow 0^+} \phi^\prime(x), \quad \phi^\prime_1 := \lim_{x \rightarrow 1^-} \phi^\prime(x),
\end{equation}
which exist (possibly taking the values $-\infty$ or $+\infty$) due to convexity.  These limiting derivatives, especially $\phi_0^\prime$, determine whether the optimal transport plan is strictly positive or admits sparsity, a structural property that is critical for the well-posedness of the inverse problem. Furthermore, \Cref{asmp:breg} guarantees that the Fenchel conjugate $\phi^*$ has domain $\operatorname{dom} \phi^*$ with $\operatorname{int}(\operatorname{dom} \phi^*)$ also an interval, and that $\phi^*$ is $C^1$ on $\operatorname{int}(\operatorname{dom} \phi^*)$.


\section{Bregman Regularized Optimal Transport}\label{sec:forward_ot}
We first recall some basic properties of Bregman-regularized optimal transport (OT) problem. Given a Bregman function $\phi$, denote $\psi = \phi^*$. As a consequence, we have $\nabla\psi = (\nabla\phi)^{-1}$. The Bregman-regularized OT problem can be formulated as follows:
\begin{equation}
    \min_{X\in \mathcal{U}(\mu,\nu)} \langle C,X\rangle+\gamma\phi(X),
\end{equation}
with $\gamma >0$. 
When $\phi$ is strictly convex, the solution of \eqref{eqn:breg_ot} is unique and the inverse problem is well defined. Here we list some examples of Bregman function in \Cref{tab:bregman_func}.
\begin{table}[]
    \centering
    \begin{tabular}{llll}
\hline$\phi(x) $ & $B_\phi(x \| y) $ & $\operatorname{dom} \phi$ & $\operatorname{dom} \psi$ \\
\hline Boltzmann-Shannon entropy & Kullback-Leibler divergence & & \\
$x \log x-x+1$ & $x \log \frac{x}{y}-x+y$ & $\mathbb{R}_{+}$ & $\mathbb{R}$ \\ \hline \vspace{-2mm} \\
Burg entropy & Itakura-Saito divergence & & \\
$x-\log x-1$ & $\frac{x}{y}-\log \frac{x}{y}-1$ & $\mathbb{R}_{++}$ & $(-\infty, 1)$  \\ \hline  \vspace{-2mm} \\
Fermi-Dirac entropy & Logistic loss function & & \\
$x \log x+(1-x) \log (1-x)$ & $x \log \frac{x}{y}+(1-x) \log \frac{1-x}{1-y}$ & {$[0,1]$} & $\mathbb{R}$ \\ \hline  \vspace{-2mm} \\
$\beta$-potentials $(0<\beta<1)$ & $\beta$-divergences & & \\
$\frac{1}{\beta(\beta-1)}\left(x^\beta-\beta x+\beta-1\right)$ & $\frac{1}{\beta(\beta-1)}\left(x^\beta+(\beta-1) y^\beta-\beta x y^{\beta-1}\right)$ & $\mathbb{R}_{+}$ & $\left(-\infty, \frac{1}{1-\beta}\right)$
\\
\hline
\end{tabular}
    \caption{Examples of the Bregman function.}
    \label{tab:bregman_func}
\end{table}

To characterize the solution and its relation to the cost matrix, we summarize the duality results and optimality conditions in the following lemma, the proof of which can be found in the Appendix.
\begin{lem}[Optimality Conditions for Forward Bregman OT]
\label{lem:forward_kkt}
Suppose Assumption \ref{asmp:breg} holds.
\begin{enumerate}
    \item[(a)] The Lagrange dual of \eqref{eqn:breg_ot} is given by the unconstrained problem:
    \begin{equation}\label{eqn:breg_dual_ot}
        \min_{u,v\in\mathbb{R}^n} \Psi(u,v):= \left\langle u\oplus v - C, \left(\nabla\psi\left(\frac{u\oplus v - C}{\gamma}\right)\right)_+ \right\rangle-\langle u,\mu\rangle-\langle v,\nu\rangle-\gamma\phi\left(\left(\nabla\psi\left(\frac{u\oplus v - C}{\gamma}\right)\right)_+\right).
    \end{equation}  
    \item[(b)]A transport plan $X \in \mathcal{U}(\mu, \nu)$ and dual potentials $u, v \in \mathbb{R}^n$ are optimal if and only if they satisfy the following KKT condition for all $i,j$:
    \begin{equation}\label{eqn:KKT_bregman}
         \left\{
    \begin{aligned}
        &X^C = \left(\nabla\psi\left(\frac{u\oplus v - C}{\gamma}\right)\right)_+, \\
        &X^C\in\mathcal{U}(\mu,\nu).
    \end{aligned}
    \right.
    \end{equation}
    \item[(c)] If the regularizer satisfies $\phi'_0 = -\infty$, the KKT condition in (b) simplifies to:
    \begin{equation}\label{eqn:KKT_bregman_positive}
        \left\{
    \begin{aligned}
        &X^C = \nabla\psi\left(\frac{u\oplus v - C}{\gamma}\right), \\
        &X^C\in\mathcal{U}(\mu,\nu),
    \end{aligned}
    \right.
    \end{equation}
    In particular, the dual problem can be simplified into:
    \begin{equation}\label{eqn:breg_dual_ot_postive}
        \min_{u,v\in\mathbb{R}^n} \Psi(u,v):= \gamma\psi\left(\frac{u\oplus v - C}{\gamma}\right) - \left\langle u,\mu\right\rangle - \left\langle v,\nu\right\rangle. 
    \end{equation}
\end{enumerate}
\end{lem}

With the basic properties of the forward problem presented, including the important KKT conditions, we now turn to the study of the well-posedness of the inverse Bregman-regularized OT problem.

\section{Well-Posedness of Inverse Bregman-Regularized OT Problems}\label{sec:well_posed}
To analyze the well-posedness of the inverse Bregman-regularized optimal transport problem, we begin by examining the properties of the solution to the forward problem. Consider the forward mapping
\[
\mathcal{F}(C) := \arg\min_{X \in \mathcal{U}(\mu,\nu)} \left\{ \langle C, X \rangle + \gamma\,\phi(X) \right\}.
\]
The study of the inverse problem's well-posedness can thus be reduced to analyzing the mapping properties of $\mathcal{F}(C)$. In particular, the \emph{existence} of a solution to the inverse problem, given an observed transport plan, is tied to characterizing the \emph{range} of $\mathcal{F}(C)$. The \emph{uniqueness} of the recovered cost matrix depends critically on whether $\mathcal{F}(C)$ is \emph{injective}. Likewise, the \emph{stability} of the inverse problem can be assessed by quantifying how perturbations in the optimal plan $\hat{X}$ affects the cost matrix $C = \mathcal{F}^{-1}(\hat{X})$.

Importantly, all of these mapping properties are influenced by the domain of the forward operator $\mathcal{F}$. That is, they depend on the structural assumptions placed on the set of admissible cost matrices. In this section, we examine the behavior of the forward map $\mathcal{F}$ under different domain choices and for various regularizers $\phi$. We consider two primary cases: first, when \textit{C} is in the general space of non-negative matrices (Section \ref{subsec:case1}), and second, when \textit{C} is restricted to the more structured space of symmetric matrices with zero diagonals (Section \ref{subsec:case2}).


\subsection{Case 1: Domain of \texorpdfstring{$\mathcal{F}(C)$}{} is  \texorpdfstring{$\mathbb{R}^{n\times n}$}{} }\label{subsec:case1}
First of all, we investigate the existence of the cost matrix $C$ given any observed transport plan $\hat{X}\in\mathbb{R}_+^{n\times n}$. Define the target set
$$
\mathcal{T}_\phi(\mu,\nu):=\begin{cases}
\mathcal{U}(\mu, \nu) & \text{if } \phi'_0, \phi'_1 \text{ are finite}\\
\mathcal{U}(\mu, \nu) \cap (0, 1]^{n \times n} & \text{if } \phi'_0=-\infty, \phi'_1 \text{ finite}\\
\mathcal{U}(\mu, \nu) \cap [0, 1)^{n \times n} & \text{if } \phi'_0 \text{ finite}, \phi'_1=+\infty\\
\mathcal{U}(\mu, \nu) \cap (0, 1)^{n \times n} & \text{if } \phi'_0=-\infty, \phi'_1=+\infty
\end{cases}
$$
where $\phi'_0$ and $\phi'_1$ are given in~\eqref{eq:limit_derivative}.
To better illustrate the target set $\mathcal{T}_\phi(\mu,\nu)$ for different Bregman regularizer $\phi$, we list some examples in \Cref{tab:target_set}.
\begin{table}
    \centering
    \begin{tabular}{cccc}
    \hline
       $\phi(x)$  & $\phi^\prime_0$ & $\phi^\prime_1$ & $\mathcal{T}_\phi(\mu,\nu)$\\
    \hline
       $\frac{1}{2}x^2$  & 0 & 1 & $\mathcal{U}(\mu, \nu)$\\
        $x\log x-x+1$ & $-\infty$ & -1 & $\mathcal{U}(\mu, \nu) \cap (0, 1]^{n \times n}$\\
       $x\log x+(1-x)\log(1-x)$  & $-\infty$ & $+\infty$ & $\mathcal{U}(\mu, \nu) \cap (0, 1)^{n \times n}$\\
    \hline
    \end{tabular}
    \caption{Examples of the target set $\mathcal{T}_\phi(\mu,\nu)$ for some common choices of Bregman function.}
    \label{tab:target_set}
\end{table}

We have the following existence result:
\begin{prop}\label{prop:-1}
    Assume the \Cref{asmp:breg} holds. The mapping $\mathcal{F}: \mathbb{R}^{n\times n}\rightarrow \mathcal{T}_\phi(\mu,\nu)$ is a surjection.
\end{prop}
\begin{proof}
    Clearly the range of $\mathcal{F}$ for $C\geq0$ is contained in $\mathcal{T}_\phi(\mu,\nu)$. Without loss of generality, assume that both $\phi^\prime_0$ and $\phi^\prime_1$ are finite. Given any $\hat{X}\in \mathcal{T}_\phi(\mu,\nu)$. Define $E:=\{(i,j):\hat{X}_{ij} > 0\}$. The surjectivity of $\mathcal{F}$ is equivalent to find $u,v,C$ such that $C\in\mathbb{R}^{n\times n}$ and, by the KKT condition~\eqref{eqn:KKT_bregman}, 
    \begin{equation}\label{eqn:-1}
    \left\{ \begin{aligned}
        &u_i+v_j-C_{ij} = \gamma\nabla\phi(\hat{X}_{ij}),  & 
        \forall (i, j)\in E, \\
        & u_i+v_j-C_{ij} \leq \gamma\phi^\prime_0, & \text{otherwise.}
    \end{aligned}
        \right.
    \end{equation}
    Indeed, let $K:= \max\limits_{i,j}\gamma\nabla\phi(\hat{X}_{ij})$. Define 
    \[
    C_{ij} := \left\{
    \begin{aligned}
        &K-\gamma\nabla\phi(\hat{X}_{ij})\, , \quad &\forall (i,j)\in E,\\
        &K-\gamma\phi^\prime_0 \, , \quad &\text{otherwise.}
    \end{aligned}
    \right.
    \]
    Let $u_i = v_j = \frac{K}{2}$ for all $i,j$. One can verify that $C\in\mathbb{R}^{n\times n}$ by the convexity of $\phi$ and $u,v,C$ satisfies \eqref{eqn:-1}.
\end{proof}

\begin{rmk}
As shown in \Cref{prop:-1}, for any observed transport plan $\hat{X} \in \mathcal{U}(\mu,\nu)$, there always exists a nonnegative cost matrix $C$ such that $\hat{X}$ is the unique solution to the regularized optimal transport problem. This result holds provided that every entry of $\hat{X}$ lies within the effective domain of $\nabla\phi$. 
Importantly, the conclusion of \Cref{prop:-1} applies to all commonly used regularizers listed in \Cref{tab:bregman_func} and \Cref{tab:target_set}.

\end{rmk}

Nevertheless, the mapping $\mathcal{F}$ is generally not injective. This is because, for any cost matrix $C \in \mathbb{R}^{n\times n}$ and any vectors $a, b \in \mathbb{R}^n$ such that $C + a \oplus b \in \mathbb{R}^{n\times n}$, we have $\mathcal{F}(C) = \mathcal{F}(C + a \oplus b)$, since for all transport plans $X$, it holds that
\[
\langle C + a \oplus b, X \rangle = \langle C, X \rangle + \langle a, \mu \rangle + \langle b, \nu \rangle.
\]
Yet, when $\phi_0^\prime = -\infty$, one can show that $\mathcal{F}$ becomes injective up to additive transformations of the form $a \oplus b$. To formalize this, define an equivalence relation $\sim$ on $\mathbb{R}^{n \times n}$ by declaring $C \sim C'$ if there exist vectors $a, b \in \mathbb{R}^n$ such that $C' = C + a \oplus b$. Since $\sim$ is an equivalence relation, we denote the equivalence class of $C$ by
\[
[C] := \left\{ C' \in \mathbb{R}^{n \times n} : C' \sim C \right\}.
\]
Let $\mathbb{R}^{n \times n}/\sim$ be the corresponding quotient space. Define the induced mapping $\bar{\mathcal{F}} : \mathbb{R}^{n \times n}/\sim \,\rightarrow\, \mathcal{T}_\phi(\mu, \nu)$ by
\[
\bar{\mathcal{F}}([C]) := \mathcal{F}(C),
\]
which is well-defined due to the invariance of $\mathcal{F}$ under shifts by $a \oplus b$.

We now state the following proposition:

\begin{prop}\label{prop:bijection_whole_space}
    Assume that \Cref{asmp:breg} holds and $\phi_0^\prime = -\infty$. Then the map $\bar{\mathcal{F}}:\mathbb{R}^{n\times n}/\sim \, \rightarrow\, \mathcal{T}_\phi(\mu,\nu)$ is a bijection. In particular, for any $\hat{X}\in\mathcal{T}_\phi(\mu,\nu)$, the preimage of $\mathcal{F}$ in the whole space is a $2n-1$ dimensional hyperplane defined by 
    \begin{equation}\label{eq:preimage}
    \mathcal{F}^{-1}(\hat{X}) = \left\{C\in\mathbb{R}^{n\times n} : C_{ij} = \gamma\left(\max_{ij}\left\{\nabla\phi(\hat{X}_{ij})\right\}-\nabla\phi(\hat{X}_{ij}) \right)+a_i+b_j, \,\,\forall a_i,b_j\in\mathbb{R}, \,\,\forall 1\leq i,j\leq n\right\}.
    \end{equation}
\end{prop}

\begin{proof}
    For the bijectivity of $\bar{\mathcal{F}}$, we only need to show the injectivity of $\bar{\mathcal{F}}$ since its surjectivity is a direct consequence of the surjectivity of $\mathcal{F}$. Given $C^1,C^2\in\mathbb{R}^{n\times n}$, the proof of injectivity is reduced to proving $\mathcal{F}(C^1) = \mathcal{F}(C^2) \implies C^1\sim C^2$. Denote $\hat{X} := \mathcal{F}(C^1) = \mathcal{F}(C^2)$. Since we assume $\phi_0^\prime = -\infty$, the KKT condition in~\eqref{eqn:KKT_bregman} implies that there exists vectors $u^1,v^1,u^2,v^2\in\mathbb{R}^n$, such that for all $i,j$, we have
    $$
    u_i^1+v_j^1-C_{ij}^1 = \gamma\nabla\phi(\hat{X}_{ij}) = u_i^2+v_j^2-C_{ij}^2.
    $$
    Rearrange the variables and we have 
    $$
    C_{ij}^2 = C_{ij}^1 + (u_i^2-u_i^1)+(v_j^2-v_j^1).
    $$
    Denote $a = u^2-u^1$ and $b = v^2-v^1$. We therefore obtain $C^2 = C^1+a\oplus b$,  which implies that $C^1\sim C^2$ and yields the injectivity of $\bar{\mathcal{F}}$. 
    
    Now given $\hat{X}\in\mathcal{T}_\phi(\mu,\nu)$, define $\hat{C}\in\mathbb{R}_+^{n\times n}$ with $\hat{C}_{ij} = \gamma\left(\max_{ij}\{\nabla\phi(\hat{X}_{ij})\}-\nabla\phi(\hat{X}_{ij})\right)$. The proof of \Cref{prop:-1} implies that $\mathcal{F}(\hat{C}) = \hat{X}$. As a result, \eqref{eq:preimage} holds. 
    Clearly, $\operatorname{dim}([\hat{C}]) = \operatorname{dim}\{a\oplus b:a,b\in\mathbb{R}^n\} = 2n-1$.
\end{proof}

\begin{rmk}
    When the Bregman regularizer $\phi$ is chosen as Boltzmann-Shannon entropy function $\phi(x) = x\log x-x+1$, 
    \Cref{prop:bijection_whole_space} is reduced to \cite[Theorem 3.6]{chiu2022discrete}. It is worth noting that the assumption of $\phi_0^\prime = -\infty$ is essential in establishing the injectivity of $\bar{\mathcal{F}}$. As we will demonstrate later in \Cref{prop:non_injection}, injectivity may be lost when this condition is not met.  
\end{rmk}

\subsection{Case 2: Domain of \texorpdfstring{$\mathcal{F}(C)$}{} is \texorpdfstring{$S_h$}{} }\label{subsec:case2}
In many OT applications, the cost matrix is typically derived from a distance or metric. Consequently, in IOT problems, it is natural to seek a cost matrix that corresponds to a well-defined mathematical metric. Such matrices are usually nonnegative, symmetric, and have zero diagonal entries; that is, $C \in S_h$. The mapping properties of $\mathcal{F}$ are particularly well-behaved over the set $S_h$ when the condition $\phi_0^\prime = -\infty$ holds. This condition is satisfied by the entropy function $\phi(x) = x\log x-x+1$, one of the most commonly used regularizers in OT problems.


For any $X\in \mathcal{I}:= I^{n \times n}$ with $I = \operatorname{dom}\phi$, define a mapping $G_{\phi}^\gamma :\mathcal{I} \rightarrow \mathbb{R}^{n\times n} $ such that 
$$
\left( G_{\phi}^\gamma(X) \right)_{ij} = \frac{\gamma}{2}\left( \nabla\phi(X_{ii})+\nabla\phi(X_{jj}) - \nabla\phi(X_{ij})-\nabla\phi(X_{ji}) \right).
$$
Define 
$$
U_\phi:=\{X\in\mathcal{T}_\phi(\mu,\nu): G_{\phi}^\gamma(X)_{ij}\geq0, \forall 1\leq i,j\leq n\}.
$$
We have the following existence and uniqueness result:
\begin{prop}\label{prop:0}
    Suppose \Cref{asmp:breg} holds and $\phi^\prime_0 = -\infty$. We have that $U_\phi$ is nonempty and the restricted forward mapping $\mathcal{F}: S_h\rightarrow U_\phi$ is a bijection. Furthermore, $\forall \hat{X}\in U_\phi$, we have $\mathcal{F}|_{S_h}^{-1}(\hat{X}) = G_{\phi}^\gamma(\hat{X})$.
\end{prop}
\begin{proof}
    We first show that $\mathcal{F}(S_h) \subset U_\phi$ and $U_\phi$ is thus nonempty. For any $C\in S_h$, denote $\hat{X} = \mathcal{F}(C)$. Since for all $i,j$, $\hat{X}_{ij} > 0$, by the KKT condition~\eqref{eqn:KKT_bregman_positive}, there exist $u,v\in\mathbb{R}^n$ such that 
    \begin{numcases}{}
        &$u_i+v_i = \gamma\nabla\phi(\hat{X}_{ii}),$ \label{eqn:1} \\
         &$u_i+v_j-C_{ij} = \gamma\nabla\phi(\hat{X}_{ij}),  \ \forall i \neq j $\label{eqn:2}
    \end{numcases}
    From \eqref{eqn:1}, we obtain $v_i = \gamma\nabla\phi(\hat{X}_{ii}) - u_i$. Then by substituting it into \eqref{eqn:2}, we obtain
    \begin{equation}\label{eqn:3}
        u_i+\gamma\nabla\phi(\hat{X}_{jj}) - u_j-C_{ij} = \gamma\nabla\phi(\hat{X}_{ij}).
    \end{equation}
    Interchanging $i$ and $j$ in \eqref{eqn:3} and adding the resulting equation back to \eqref{eqn:3}, we then obtain the following result given the symmetry of $C$
    \begin{equation*}
    \gamma\nabla\phi(\hat{X}_{jj})+\gamma\nabla\phi(\hat{X}_{ii})-2C_{ij} = \gamma\nabla\phi(\hat{X}_{ij})+\gamma\nabla\phi(\hat{X}_{ji}).
    \end{equation*}
    Recall that $C_{ij}\geq 0$. Therefore, we must have $\nabla\phi(\hat{X}_{ii})+\nabla\phi(\hat{X}_{jj})\geq \nabla\phi(\hat{X}_{ij})+\nabla\phi(\hat{X}_{ji})$ and $C_{ij} $ being the $ij$-th entry of $G_{\phi}^\gamma(X)$. As a result, $\hat{X} \in U_\phi$.
    
    Next, we show that $\mathcal{F}: S_h\rightarrow U_\phi$ is a surjection, i.e., $U_\phi \subset \mathcal{F}(S_h)$. Given any $\hat{X}\in U_\phi$, we define the matrix $\hat{C}\in \mathbb{R}^{n\times n}$ whose $ij$-th entry is
    $$\hat{C}_{ij}:= \frac{\gamma}{2}\left( \nabla\phi(\hat{X}_{ii})+\nabla\phi(\hat{X}_{jj}) - \nabla\phi(\hat{X}_{ij})-\nabla\phi(\hat{X}_{ji}) \right).$$ 
    Obviously, $\hat{C}\in S_h$. Choose $u\in\mathbb{R}^n$ with $$u_i = \frac{\gamma}{2}\left( \nabla\phi(\hat{X}_{ii})+\nabla\phi(\hat{X}_{1i}) - \nabla\phi(\hat{X}_{11})-\nabla\phi(\hat{X}_{i1}) \right)$$ and choose $v\in\mathbb{R}^n$ with $v_i = \gamma\nabla\phi(\hat{X}_{ii}) - u_i$. One can easily check such $u,v,\hat{C}$ and $\hat{X}$ satisfy the KKT conditions in~\eqref{eqn:1} and~\eqref{eqn:2}. Therefore, $\hat{X} = \mathcal{F}(\hat{C})$ and $U_\phi \subset \mathcal{F}(S_h)$.

    Finally, if for $\hat{C},\Tilde{C}\in S_h$ we have $\mathcal{F}(\hat{C}) = \hat{X} = \mathcal{F}(\Tilde{C})$, by the argument above one must have $\hat{C}_{ij} = \frac{\gamma}{2}\left( \nabla\phi(\hat{X}_{ii})+\nabla\phi(\hat{X}_{jj}) - \nabla\phi(\hat{X}_{ij})-\nabla\phi(\hat{X}_{ji}) \right) = \Tilde{C}_{ij}$. Hence, the restricted forward mapping $\mathcal{F}: S_h\rightarrow U_\phi$ is an injection.
\end{proof}

\Cref{prop:0} states that for any observed transport plan $\hat{X} \in U_\phi$, one can recover a nonnegative, symmetric cost matrix with zero diagonal that is unique within the set $S_h$. This result holds in particular for the entropy regularizer $\phi(x) = x \log x - x + 1$. Furthermore, this uniqueness result immediately yields the following stability guarantee:

\begin{prop}\label{prop:stab}
    Suppose \Cref{asmp:breg} holds and $\phi^\prime_0 = -\infty$. Given $\hat{X},\Tilde{X}\in U_\phi$, let $\hat{C} = \mathcal{F}|_{S_h}^{-1}(\hat{X})$ and $\Tilde{C} = \mathcal{F}|_{S_h}^{-1}(\Tilde{X})$. Then
    \begin{equation}\label{eqn:stab}
        \| \hat{C} - \Tilde{C} \|_{\infty} \leq 2\gamma\| \nabla\phi(\hat{X}) - \nabla\phi(\Tilde{X}) \|_{\infty}.
    \end{equation}
\end{prop}
\begin{proof}
    Note that
    \[
    \hat{C}_{ij} = \frac{\gamma}{2}\left( \nabla\phi(X_{ii})+\nabla\phi(X_{jj}) - \nabla\phi(X_{ij})-\nabla\phi(X_{ji}) \right).
    \]
    So we have for any $i,j$,
    \[
    \begin{aligned}
        &|\hat{C}_{ij} - \Tilde{C}_{ij}| \\
        \leq&\frac{\gamma}{2}\left( |\nabla\phi(\hat{X}_{jj}) - \nabla\phi(\Tilde{X}_{jj})| + |\nabla\phi(\hat{X}_{ii}) - \nabla\phi(\Tilde{X}_{ii})| + |\nabla\phi(\hat{X}_{ij}) - \nabla\phi(\Tilde{X}_{ij})| + |\nabla\phi(\hat{X}_{ji}) - \nabla\phi(\Tilde{X}_{ji})| \right) \\
        \leq& 2\gamma\| \nabla\phi(\hat{X}) - \nabla\phi(\Tilde{X}) \|_{\infty}
    \end{aligned}
    \]
    As a result, \eqref{eqn:stab} holds.
\end{proof}

\begin{rmk}
    To further explain the results from \Cref{prop:0} and \Cref{prop:stab}, we consider the entropy regularizer $\phi(x) = x \log x - x + 1$ as a concrete example. In this case, the inverse mapping restricted to $S_h$ takes the form
\[
\mathcal{F}|_{S_h}^{-1}(\hat{X})_{ij} = G_{\phi}^\gamma(\hat{X})_{ij} = \log\sqrt{\frac{X_{ii} X_{jj}}{X_{ij}X_{ji}}},
\]
and the set of admissible transport plans becomes $
U_\phi = \left\{ X \in \mathcal{T}_\phi(\mu,\nu) : X_{ii} X_{jj} \geq X_{ij} X_{ji}, \quad \forall\, 1 \leq i,j \leq n \right\}$. Under this setting, the stability result~\eqref{eqn:stab} specializes to $
    \| \hat{C} - \Tilde{C} \|_{\infty} \leq 2\gamma\| \log\hat{X} - \log\Tilde{X} \|_{\infty}$.
\end{rmk}

Now given a fixed vector $w\in\mathbb{R}^n$, we relax $S_h$ into $S_h^w := \{C\in\mathbb{R}_+^{n\times n}: C=C^\top, \, \operatorname{diag}(C) = w\}$. We have the following similar result:
\begin{coro}
    Suppose \Cref{asmp:breg} holds and $\phi^\prime_0 = -\infty$. Then $U_\phi^w:=\{ X\in\mathcal{T}_\phi(\mu,\nu):G_{\phi}^\gamma(X)_{ij} + (w_i+w_j)/2 \geq 0, \ \forall 1\leq i,j \leq n\}$ is nonempty and the restricted forward mapping $\mathcal{F}: S_h^w\rightarrow U_\phi^w$ is a bijection. And, $\forall \hat{X}\in U_\phi^w$, we have $\mathcal{F}|_{S_h^w}^{-1}(\hat{X}) = G_{\phi}^\gamma(\hat{X})+ \frac{1}{2}w\oplus w$. Furthermore, given $\hat{X},\Tilde{X}\in U_\phi^w$, let $\hat{C} = \mathcal{F}|_{S_h^w}^{-1}(\hat{X})$ and $\Tilde{C} = \mathcal{F}|_{S_h^w}^{-1}(\Tilde{X})$. Then~\eqref{eqn:stab} holds.
\end{coro}

\begin{proof}
    By the KKT condition, there exist $u,v\in\mathbb{R}^n$ such that 
    \begin{numcases}{}
        &$u_i+v_i -w_i = \gamma\nabla\phi(\hat{X}_{ii}),$ \notag\\
         &$u_i+v_j-C_{ij} = \gamma\nabla\phi(\hat{X}_{ij}),  \ \forall i \neq j $,\notag
    \end{numcases}
    As a result,
    \begin{equation*}      \gamma\nabla\phi(\hat{X}_{jj})+\gamma\nabla\phi(\hat{X}_{ii})-2C_{ij} +w_i+w_j = \gamma\nabla\phi(\hat{X}_{ij})+\gamma\nabla\phi(\hat{X}_{ji}).
    \end{equation*}
    So we have $C_{ij} = G_{\phi}^\gamma(\hat{X})_{ij} + (w_i+w_j)/2$. The remainder of the proof follows similar arguments to those used in the proofs of~\Cref{prop:0} and~\Cref{prop:stab}.
\end{proof}

\begin{rmk}
The requirement of a fixed diagonal is essential for the injectivity of $\mathcal{F}$. Suppose we only assume that $C$ is a nonnegative symmetric matrix. Then, for any vector $w \in \mathbb{R}_+^n$, the matrix $C + w \oplus w$ is also nonnegative and symmetric. Moreover, for any transport plan $X \in \mathcal{U}(\mu,\nu)$, we have $
\langle C + w \oplus w, X \rangle = \langle C, X \rangle + \langle w, \mu \rangle + \langle w, \nu \rangle$. It follows that $\mathcal{F}(C + w \oplus w) = \mathcal{F}(C)$ for all $w \in \mathbb{R}_+^n$. However, as shown in \Cref{prop:bijection_whole_space}, when $\mathcal{F}$ is restricted to the space of nonnegative symmetric matrices, it remains injective up to additive transformations of the form $w \oplus w$.
\end{rmk}

If we further restrict the domain of admissible cost matrices, we can obtain analogous results. Define the set of all matrices satisfying the triangle inequality:
\[
\operatorname{MC} := \left\{ C \in S_h : C_{ij} \leq C_{ik} + C_{kj},\ \forall\, i,j,k \right\},
\]
and the set of all Euclidean distance matrices:
\[
\operatorname{ED} := \left\{ C \in \mathbb{R}^{n \times n} : \text{there exist } d \in \mathbb{N} \text{ and } x_1, \ldots, x_n \in \mathbb{R}^d \text{ such that } C_{ij} = \|x_i - x_j\|^2 \right\}.
\]
It is known that $\operatorname{ED}$ admits an equivalent characterization. Let
\[
\mathcal{K}_+^n := \left\{ A \in \mathcal{S}^n : x^\top A x \geq 0 \text{ for all } x \in \mathbb{R}^n \text{ satisfying } \mathbf{1}^\top x = 0 \right\}.
\]
As shown in~\cite{qi2013semismooth}, we have the identity
\[
\operatorname{ED} = S_h \cap (-\mathcal{K}_+^n).
\]
Now, define two subsets of $U_\phi$:
\[
V_\phi := \left\{ X \in U_\phi : 2\nabla\phi(X_{kk}) + \nabla\phi(X_{ij}) + \nabla\phi(X_{ji}) \geq \nabla\phi(X_{ik}) + \nabla\phi(X_{ki}) + \nabla\phi(X_{jk}) + \nabla\phi(X_{kj}),\ \forall\, i,j,k \right\},
\]
and
\[
W_\phi := \left\{ X \in U_\phi : -G_\phi^\gamma(X) \in \mathcal{K}_+^n \right\}.
\]
With these definitions, we now state the following corollaries:
\begin{coro}
    Suppose \Cref{asmp:breg} holds and $\phi^\prime_0 = -\infty$. Then $V_\phi$ is nonempty and the restricted forward mapping $\mathcal{F}: \operatorname{MC}\rightarrow V_\phi$ is a bijection. In addition, $\forall \hat{X}\in V_\phi$, we have $\mathcal{F}|_{\operatorname{MC}}^{-1}(\hat{X}) = G_\phi^\gamma(\hat{X})$. Furthermore, given $\hat{X},\Tilde{X}\in V_\phi$, let $\hat{C} = \mathcal{F}|_{\operatorname{MC}}^{-1}(\hat{X})$ and $\Tilde{C} = \mathcal{F}|_{\operatorname{MC}}^{-1}(\Tilde{X})$. Then~\eqref{eqn:stab} holds.
\end{coro}

\begin{coro}
     Suppose \Cref{asmp:breg} holds and $\phi^\prime_0 = -\infty$. Then $W_\phi$ is nonempty and the restricted forward mapping $\mathcal{F}: \operatorname{ED}\rightarrow W_\phi$ is a bijection. Moreover, $\forall \hat{X}\in W_\phi$, we have $\mathcal{F}|_{\operatorname{ED}}^{-1}(\hat{X}) = G_\phi^\gamma(\hat{X})$. Furthermore, given $\hat{X},\Tilde{X}\in W_\phi$, let $\hat{C} = \mathcal{F}|_{\operatorname{ED}}^{-1}(\hat{X})$ and $\Tilde{C} = \mathcal{F}|_{\operatorname{ED}}^{-1}(\Tilde{X})$. Then~\eqref{eqn:stab} holds.
\end{coro}

However, when $\phi_0^\prime$ is finite (for example, in the case of the $\ell^2$ regularizer $\phi(x) = \frac{1}{2}x^2$), the mapping $\mathcal{F}$ is no longer injective, and the uniqueness of the recovered cost matrix fails. This is formalized in the following proposition:
\begin{prop}\label{prop:non_injection}
Suppose \Cref{asmp:breg} holds and $\phi_0^\prime$ is finite. Let $\hat{C}$ be a given cost matrix and define $\hat{X} := \mathcal{F}(\hat{C})$. If there exists $1 \leq i,j \leq n$ such that $\hat{X}_{ij} = 0$, then for all $\lambda > 0$, we have
\[
\mathcal{F}(\hat{C}) = \hat{X} = \mathcal{F}(\hat{C} + \lambda E^{(i,j)}),
\]
where $E^{(i,j)}$ denotes the $n \times n$ matrix with  $1$ in the $(i,j)$-th entry and zeros elsewhere.
\end{prop}
\begin{proof}
    By the KKT condition~\eqref{eqn:KKT_bregman}, there exist optimal dual variables $\hat{u},\hat{v}$ corresponding to $\hat{X}$ such that $\hat{X} = \left(\nabla\psi\left(\frac{\hat{u}\oplus \hat{v} - \hat{C}}{\gamma}\right)\right)_+$. Together with $\hat{X}_{ij} = 0$, we have that
    \[
    \psi^\prime\left(\frac{\hat{u}_i + \hat{v}_j - \hat{C}_{ij}}{\gamma}\right) \leq 0.
    \]
    The convexity of $\phi$ implies the monotonicity of $\phi^\prime$. Hence, we have
    \[
    \hat{u}_i + \hat{v}_j \leq \hat{C}_{ij} + \gamma\phi^\prime_0.
    \]
     For any $\lambda > 0$,
    \[
    \hat{u}_i + \hat{v}_j < \lambda + \hat{C}_{ij} + \gamma\phi^\prime_0,
    \]
    which yields 
    \[
    \psi^\prime\left(\frac{\hat{u}_i + \hat{v}_j - (\hat{C}_{ij}+\lambda)}{\gamma}\right) \leq 0.
    \]
    Consequently, $(\hat{X},\hat{u},\hat{v})$ satisfies the KKT condition of the following optimization problem
    \[
    \min_{X\in \mathcal{U}(\mu,\nu)} \langle \hat{C}+\lambda E^{(i,j)},X\rangle+\gamma\phi(X),
    \]
    which implies $\hat{X}$ is also the unique optimal solution to the above optimization problem and $\hat{X} = \mathcal{F}(\hat{C}+\lambda E^{(i,j)})$.
\end{proof}
Even if we further require the cost matrix $C\in \operatorname{ED}$, the uniqueness of IOT does not necessarily hold, as the following example suggests.
\begin{exm}
    Choose $0\leq a\leq b\leq 1$. Define $\hat{X} := \begin{pmatrix}
        a & 0 \\
        b-a& 1-b
    \end{pmatrix}$
    , $\mu = \begin{pmatrix} a\\ 1-a  \end{pmatrix}$ and $\nu = \begin{pmatrix} b\\ 1-b  \end{pmatrix}$. Choose any $\gamma \geq \max\{\frac{1}{2} +a-b,0\}$ and denote $C = \begin{pmatrix}
        0 & \gamma \\
        \gamma & 0
    \end{pmatrix}$. Direct computation shows that $\hat{X}$ is the unique solution to the following regularized optimal transport problem for all $\gamma \geq \max\{\frac{1}{2} +a-b,0\}$:
    \[
    \min_{X\in\mathcal{U}(\mu,\nu)} \langle C,X\rangle + \frac{1}{2}\|X\|_F^2.
    \]
\end{exm}
As demonstrated in \Cref{prop:non_injection}, the sparsity of the observed transport plan poses significant challenges for uniquely recovering the cost matrix $C$. This naturally leads to the following question: if we observe multiple sparse transport plans $X_1, \dots, X_n$, whose combined support covers every entry, can the corresponding cost matrix $C$ be uniquely determined, provided that such a cost matrix exists? Unfortunately, even under these extra conditions, uniqueness of $C$ is not guaranteed. This remains true even if we restrict $C$ to lie within $\operatorname{ED}$, as evidenced by a counterexample.
\begin{exm}
     Choose $X_1 = \begin{pmatrix}
        0.25 & 0.5 \\
        0 & 0.25
    \end{pmatrix}$ 
    and $X_2 = \begin{pmatrix}
        0.25 & 0 \\
       0.5 & 0.25
    \end{pmatrix}$. Let $\mu_1 =\nu_2 = (0.75, 0.25)^\top, \nu_1 = \mu_2 =(0.25,0.75)^\top.$ Choose any $\gamma \geq 0$ and denote $C = \begin{pmatrix}
        0 & \gamma \\
        \gamma & 0
    \end{pmatrix}$. Direct computation shows that 
    \[
    X_1 = \argmin_{X\in\mathcal{U}(\mu_1,\nu_1)} \langle C,X\rangle + \frac{1}{2}\|X\|_F^2,\quad \text{and}\quad 
    X_2 = \argmin_{X\in\mathcal{U}(\mu_2,\nu_2)} \langle C,X\rangle + \frac{1}{2}\|X\|_F^2,
    \]
    while $\operatorname{supp}(X_1)\bigcup\operatorname{supp}(X_2) = \{(i,j): 1\leq i,j \leq 2\}$. 
\end{exm}

\section{Inverse Bregman-Regularized OT as an Optimization Problem}\label{sec:iot_model}
Having established the well-posedness of the inverse problem in the previous section, we now shift our focus from theoretical properties to computational methodology. This section addresses two key goals: first, formulating the IOT problem as a tractable, single-level convex program and analyzing the existence of its solutions; and second, developing an efficient Block Coordinate Descent (BCD) algorithm with proven convergence guarantees to solve it. 

\subsection{Single-Level Formulation and Existence}
\label{sec:single-level}

A natural approach to inferring the cost matrix $C$ from an observed transport plan $\hat{X} \in \mathcal{T}_\phi(\mu, \nu)$ is to find a cost $C$ such that the resulting transport plan $\mathcal{F}(C)$ is as close as possible to $\hat{X}$. This can be formulated as the following optimization problem:
\begin{equation}\label{eqn:breg_inverse_ot}
    \inf_{C \in \mathcal{C}} B_\phi(\hat{X} \| \mathcal{F}(C)),
\end{equation}
where $\mathcal{C}$ is a closed and convex set representing prior knowledge on the cost matrix (e.g., non-negativity or symmetry), and $B_\phi$ is the Bregman divergence corresponding to the regularizer $\phi$.

This formulation, however, presents significant challenges. Since the forward map $\mathcal{F}(C)$ is itself the solution to an optimization problem (the forward OT problem), equation~\eqref{eqn:breg_inverse_ot} is a bilevel optimization problem. Such problems are notoriously difficult to analyze and solve. For a general choice of $\mathcal{C}$, establishing the existence of an optimal solution is non-trivial, and even when a solution exists, the computational cost can be prohibitive.

To create a more tractable framework, we seek to reformulate this bilevel problem into an equivalent single-level convex program. Define the function $F(Z) := \psi(Z) - \langle Z, \hat{X} \rangle$. It is clear that $F$ is convex over $\operatorname{dom} \psi$. Building on the work of~\cite{ma2020learning}, the bilevel optimization problem~\eqref{eqn:breg_inverse_ot} can be reduced to a simpler single-level convex optimization problem, provided that $\phi_0^\prime = -\infty$, as formalized in the following proposition:
\begin{prop}\label{prop:equivalence}
    Suppose \Cref{asmp:breg} holds, $\phi^\prime_0 = -\infty$, and the observed transport plan $\hat{X}\in\mathcal{T}_\phi(\mu,\nu)$. Then the bilevel IOT problem \eqref{eqn:breg_inverse_ot} is equivalent to the following single-level convex optimization problem:
    \begin{equation}\label{eqn:breg_iot_equiv}
        \inf_{C\in\mathcal{C}, u,v\in\mathbb{R}^n} E(u,v,C):= \psi\left(\frac{u\oplus v - C}{\gamma}\right)-\left\langle \frac{u\oplus v - C}{\gamma},\hat{X}\right\rangle.
    \end{equation}
    The equivalence relation is understood in the sense that if an optimal solution to one problem exists, then so does the optimal solution to the other. Furthermore, given the existence of optimal solutions to both problems, $C^*$ solves \eqref{eqn:breg_inverse_ot} if and only if $(u^{C^*}, v^{C^*}, C^*)$ solves \eqref{eqn:breg_iot_equiv}, where $(u^{C^*}, v^{C^*})$ is the optimal solution to \eqref{eqn:breg_dual_ot} with cost $C^*$.\end{prop}
\begin{proof}
    Let $(u^{C},v^{C})$ be the optimal solution to \eqref{eqn:breg_dual_ot} with cost $C$. Then by the KKT condition \eqref{eqn:KKT_bregman_positive}, we have $X^C = \nabla\psi(\frac{u^C\oplus v^C - C}{\gamma})$. Note that $\psi = \phi^*$ and $\hat{X}\in\mathcal{U}(\mu,\nu)$. Therefore
    \begin{align}
            B_\phi(\hat{X} \| X^C) &= \phi(\hat{X}) - \phi(X^C) - \langle \nabla\phi(X^C), \hat{X} - X^C\rangle \nonumber\\
            &= \phi(\hat{X}) - \phi(X^C) - \left\langle \frac{u^C\oplus v^C - C}{\gamma}, \hat{X} - X^C\right\rangle \nonumber\\
            &=\left\langle \frac{u^C\oplus v^C - C}{\gamma},  X^C \right\rangle - \phi(X^C) - \frac{1}{\gamma}\langle u^C\oplus v^C, \hat{X} \rangle + \frac{1}{\gamma}\langle C, \hat{X} \rangle + \phi(\hat{X}) \nonumber\\
            &= \left\langle \frac{u^C\oplus v^C - C}{\gamma},  X^C \right\rangle - \phi(X^C) - \frac{1}{\gamma}\langle u^C, \mu \rangle - \frac{1}{\gamma}\langle v^C, \nu \rangle + \frac{1}{\gamma}\langle C, \hat{X} \rangle + \phi(\hat{X})  \nonumber \\
            &= \psi\left(\frac{u^C\oplus v^C - C}{\gamma}\right) - \left\langle \frac{u^C\oplus v^C - C}{\gamma} ,\hat{X}\right\rangle + \phi(\hat{X}) \nonumber \\
            &= F\left(\frac{u^C\oplus v^C - C}{\gamma}\right)+\phi(\hat{X}).\label{eqn:equivalence}
        \end{align} 
    The optimality of $u^C,v^C$ yields:
    \[
    F\left(\frac{u^C\oplus v^C - C}{\gamma}\right) = \min_{u,v}F\left(\frac{u\oplus v - C}{\gamma}\right).
    \]
    Thus, we have shown
    \[
    B_\phi(\hat{X} \| X^C) = \frac{1}{\gamma}\min_{u,v}F\left(\frac{u\oplus v - C}{\gamma}\right)+\phi(\hat{X}).
    \]
    We can now conclude that the two problems  \eqref{eqn:breg_inverse_ot} and \eqref{eqn:breg_iot_equiv} are equivalent and share the same solution $C^*$.
\end{proof}

The condition $\phi^\prime_0 = -\infty$ plays a central role in the proof of \Cref{prop:equivalence} since it ensures that the KKT condition \eqref{eqn:KKT_bregman} can be reduced to $X^C = \nabla\psi(\frac{u^C\oplus v^C-C}{\gamma})$, i.e., $\nabla\phi(X^C) = \frac{u^C\oplus v^C-C}{\gamma}$, which leads to~\eqref{eqn:equivalence}. This requirement is met by the entropy function, but not by a regularizer like the quadratic function $\phi(x) = \frac{1}{2}x^2$. 

For a general closed convex set $\mathcal{C}$, the existence of the optimal solutions to \eqref{eqn:breg_inverse_ot} and \eqref{eqn:breg_iot_equiv} is unknown. However, if $\mathcal{C} = S_h$, we have the following existence result:
\begin{prop}\label{prop:equi_existence}
    Suppose $\phi:I\rightarrow\mathbb{R}$ is a proper closed convex function and $\hat{X}_{ij}\in\operatorname{int}(\operatorname{dom}\phi), \forall i,j$. Set $\mathcal{C} = S_h$. Then the infimum of \eqref{eqn:breg_iot_equiv} is attainable. Suppose $(u^*,v^*,C^*)$ is one solution to \eqref{eqn:breg_iot_equiv}. If we further assume that the Fenchel conjugate $\psi = \phi^*$ is twice differentiable and strictly convex at  $\frac{u^*\oplus v^*-C^*}{\gamma}$, i.e. $\psi^{\prime\prime}(\frac{u^*\oplus v^*-C^*}{\gamma}) > 0$, then the optimal solution $C^*\in S_h$ is unique. 
\end{prop}

\begin{proof}
    
Note that for any scalar $a \in \mathbb{R}$, $E( u+a\mathbf{1}, v-a\mathbf{1},C) = E(u, v ,C)$, where $\mathbf{1}$ represents the vector with all one. Let $V = \{ (u, v) \in \mathbb{R}^n \times \mathbb{R}^n \mid u_n = 0 \}$. The set $V$ is a closed subspace of $\mathbb{R}^{2n}$. By enforcing $u_n=0$, we eliminate the shifting invariance of the objective function $E(u,v,C)$ since in this case, if both $(u+a\mathbf{1}, v-a\mathbf{1})$ and $(u,v)$ lie in $V$, one must have the scalar $a$ to be 0. By eliminating the shifting invariance, the problem is equivalent to finding $\inf_{(u, v ,C) \in \mathcal{C} \times V} E(u, v ,C)$ and the domain $\mathcal{D} = \mathcal{C} \times V$ is a closed set. \smallskip


\emph{Existence:} Let $E^* = \inf_{(u, v, C) \in \mathcal{D}} E(u, v, C)$. The convexity of $\psi$ implies that $E(u, v, C)$ is bounded below, and therefore $E^*$ is finite. Since $\phi$ is a proper closed convex function, we have $\psi = \phi^*$, which is continuous; thus, $E(u, v, C)$ is also continuous.

Let $(u^k, v^k, C^k)$ be a minimizing sequence in $\mathcal{D}$, meaning $(u^k, v^k, C^k) \in \mathcal{C} \times V$ and $E(u^k, v^k, C^k) \to E^*$. Since $E^*$ is finite, the sequence $E(u^k, v^k, C^k)$ is bounded. Denote $y_{ij}^k = (u_i^k + v_j^k - C_{ij}^k)/\gamma$. Note that the function $h(y) = \psi(y) - y X$ has a bounded level set if $X \in \operatorname{int}\operatorname{dom}(\phi)$ \cite[Corollary~14.2.2]{rockafellar1997convex}. Since $E(u^k, v^k, C^k)$ is bounded, each term $h(y_{ij}^k)$ must also be bounded. By assumption, $\hat{X}_{ij} \in \operatorname{int}\operatorname{dom}(\phi)$, so there exists $M > 0$ such that $|y_{ij}^k| \leq M$ for all $i, j, k$. Let $Y_{ij}^k = \gamma y_{ij}^k = u_i^k + v_j^k - C_{ij}^k$. 

The matrix $C^k$ satisfies $C^k \in \mathcal{C}$, which gives $C_{ij}^k \geq 0$, $C_{ii}^k = 0$, and $C_{ij}^k = C_{ji}^k$. 
First, $C_{ii}^k = 0 \implies u_i^k + v_i^k = Y_{ii}^k$. Hence, $u_i^k + v_i^k$ is bounded. 
Second, $C_{ij}^k = C_{ji}^k \implies (u_i^k - v_i^k) - (u_j^k - v_j^k) = Y_{ij}^k - Y_{ji}^k$. Denote $d_i^k = u_i^k - v_i^k$, which implies $d_i^k - d_j^k$ is bounded. By choosing $u_n^k = 0$, we deduce that $v_n^k = Y_{nn}^k - u_n^k = Y_{nn}^k$ and $d_n^k = -v_n^k$ are both bounded. Consequently, $d_i^k = d_n^k + (Y_{in}^k - Y_{ni}^k)$ is bounded for all $i$. Therefore, $2u_i^k = (u_i^k + v_i^k) + (u_i^k - v_i^k) = Y_{ii}^k + d_i^k$ is bounded. Thus, $u^k$ is bounded, and similarly, $v^k$ is bounded. Finally, the boundedness of $C^k$ follows from $C_{ij}^k = u_i^k + v_j^k - Y_{ij}^k$. 

To summarize, any minimizing sequence $(u^k, v^k, C^k)$ in $\mathcal{D} = \mathcal{C} \times V$ is bounded. Thus, there exists a subsequence $(u^{k_m}, v^{k_m}, C^{k_m})$ that converges to a limit $(u^*, v^*, C^*) \in \mathcal{D}$ (since $\mathcal{D}$ is closed). By the continuity of $E$, we have $E(u^{k_m}, v^{k_m}, C^{k_m}) \to E(u^*, v^*, C^*)$. Since $E(u^k, v^k, C^k) \to E^*$, it follows that $E(u^*, v^*, C^*) = E^*$. Consequently, the infimum is attained at $(u^*, v^*, C^*)$. \smallskip

\emph{Uniqueness:} Let $z = (u, v ,C)$ and $y_{ij}(z) = (u_i+v_j-C_{ij})/\gamma$. Let $d = (\Delta u, \Delta v, \Delta C)$ be a direction vector. The second-order term in the Taylor expansion of $E$ around $z$ is given by the quadratic form:
$$ Q(d) = d^\top \nabla^2 L(z) d = \sum_{i,j} \frac{1}{\gamma^2} (\psi)''(y_{ij}(z)) (\Delta u_i + \Delta v_j - \Delta C_{ij})^2. $$
For $E$ to be strictly convex, we need $Q(d) > 0$ for all $d \ne 0$ in the tangent space of $\mathcal{D}$. By assumption $\psi^{\prime\prime}(\frac{u^*\oplus v^*-C^*}{\gamma}) > 0$, we only need to show the term $(\Delta u_i + \Delta v_j - \Delta C_{ij})$ must not be zero for all $i,j$ if $d \ne 0$.

Let $d = (\Delta C, \Delta u, \Delta v)$ be a direction in the tangent space of $\mathcal{D}$ where $\Delta C$ must satisfy the constraints for $S_h$, and $\Delta u_n = 0$. First, we have $\Delta u_i + \Delta v_i = \Delta C_{ii} = 0$. Second, by symmetry $\Delta C_{ij} = \Delta C_{ji}$, we obtain $\Delta u_i + \Delta v_j = \Delta u_j + \Delta v_i$, which means $\Delta u_i - \Delta v_i = \Delta u_j - \Delta v_j, \forall i,j$. Denote $B=\Delta u_1 - \Delta v_1$. Obviously $\Delta u_i - \Delta v_i = B$ hold for all $i$. From $\Delta u_i + \Delta v_i = 0$ and $\Delta u_i - \Delta v_i = B$, we get $2 \Delta u_i = B$, so $\Delta u_i = B/2$ and $\Delta v_i = -B/2$. This means that $\Delta u = (B/2) \mathbf{1}$ and $\Delta v = (-B/2) \mathbf{1}$. 
But the condition for $V$ is $\Delta u_n = 0$. This implies $B/2 = 0$, so $B=0$. Thus, $\Delta u = 0$ and $\Delta v = 0$. Then $\Delta C_{ij} = \Delta u_i + \Delta v_j =  0$. So $\Delta C = 0$. Therefore, the only direction $d$ for which $Q(d)=0$ is $d=(0, 0, 0)$. This means that $E(u, v ,C)$ is strictly convex on $\mathcal{D}$ and therefore the minimizer $(u^*, v^* ,C^*)$ must be unique.
\end{proof}


\begin{rmk}
    When the regularizer is chosen as the Boltzmann-Shannon entropy function, the minimizers of both \eqref{eqn:breg_inverse_ot} and \eqref{eqn:breg_iot_equiv} exist and coincide. The minimizer is also unique in the setting of $\mathcal{C} = S_h$, which has been discussed in \cite{ma2020learning}. 
\end{rmk}


Having established the existence and uniqueness of the solution to the optimization problem~\eqref{eqn:breg_iot_equiv}, we now shift our focus to solving it algorithmically. Given the structure of the problem, the Block Coordinate Descent (BCD) method emerges as a natural choice. However, a key assumption in standard BCD convergence proofs is the boundedness of the level sets of the objective function, which guarantees that the sequence of iterates generated by the algorithm has at least one limit point.

In this case, the objective function $E(u, v, C)$ exhibits ``flat'' directions due to the shifting invariance in the dual variables $(u, v)$, leading to unbounded level sets. As a result, an iterative method such as BCD may produce sequences $\{(u_k, v_k, C_k)\}$ where $\|u_k\|, \|v_k\| \to \infty$, while $E(u_k, v_k, C_k)$ converges. To mitigate this issue and ensure that the BCD iterates remain within a compact region, we introduce an additional strongly convex regularization term into the optimization formulation. This leads to the slightly perturbed problem:
\begin{equation}\label{eqn:breg_iot_equiv_reg}
        \min_{u,v\in\mathbb{R}^n,C\in\mathcal{C}} E_{\lambda}(u,v,C):= E(u,v,C) + \lambda(R_1(u)+R_2(v)+R_3(C)),
\end{equation}
where $\lambda > 0$ and $R_1,R_2,R_3$ are proper continuous strongly convex functions. Denote \[S = \{(u^*,v^*,C^*): (u^*,v^*,C^*) \in \arg\min E(u,v,C)\}\] and suppose $S\neq\emptyset$. The strong convexity of $R_1,R_2,R_3$ implies the strong convexity of \eqref{eqn:breg_iot_equiv_reg} and thus ensures the existence and uniqueness of the optimal solution to \eqref{eqn:breg_iot_equiv_reg}. Let $(u^*_\lambda,v^*_\lambda,C^*_\lambda)$ be the optimal solution of $E_\lambda(u,v,C)$. Denote \[H(u,v,C) = R_1(u)+R_2(v)+R_3(C)\]. We first show that the optimal value of \eqref{eqn:breg_iot_equiv_reg} converges to the optimal value of \eqref{eqn:breg_iot_equiv}.

\begin{prop}
    Suppose $S\neq\emptyset$ and $H(u,v,C)$ is nonnegative. Denote $E^*$ the minimum of $E(u,v,C)$ and $E^*_\lambda$ the minimum of $E_{\lambda}(u,v,C)$. Then $\lim\limits_{\lambda\rightarrow 0} E_\lambda^* = E^*$. 
    Furthermore, if  we assume there exists $\epsilon > 0$ such that $\{(u^*_\lambda,v^*_\lambda,C^*_\lambda)\}|_{\lambda<\epsilon}$ is bounded, then we have 
    \[
    (u^*_\lambda,v^*_\lambda,C^*_\lambda) \stackrel{\lambda\rightarrow0}{\longrightarrow} \arg\min\{ H(u,v,C): (u,v,C) \in S \}.
    \]
\end{prop}
\begin{proof}
Choose $(u^*,v^*,C^*)\in S$. By definition we have
        \begin{align}\label{eqn:0}
        E(u^*_\lambda,v^*_\lambda,C^*_\lambda) +\lambda H(u^*_\lambda,v^*_\lambda,C^*_\lambda) 
        \leq E(u^*,v^*,C^*) +\lambda H(u^*,v^*,C^*) 
        \leq E(u^*_\lambda,v^*_\lambda,C^*_\lambda) +\lambda H(u^*,v^*,C^*).
        \end{align}
    Set $\kappa = \inf_{(u^*,v^*,C^*)\in S}H(u^*,v^*,C^*)$. We have
    \[
    0\leq E^*_\lambda - E^* \leq \lambda\kappa.
    \]
    Let $\lambda\rightarrow0$ and then $\lim_{\lambda\rightarrow 0} E_\lambda^* = E^*$. The boundedness of  $\{(u^*_\lambda,v^*_\lambda,C^*_\lambda)\}|_{\lambda<\epsilon}$ guarantees the existence of a converging subsequence. We now extract a converging subsequence $(u^*_{\lambda_k},v^*_{\lambda_k},C^*_{\lambda_k})$ of $(u^*_\lambda,v^*_\lambda,C^*_\lambda)$ with $\lambda_k\rightarrow0$ that converges to a point $(u^*_0,v^*_0,C^*_0)$.  Substitute $\lambda_k$ for $\lambda$ in \eqref{eqn:0} and let $k\rightarrow\infty$. The continuity of $E(u,v,C)$ yields $E(u^*_0,v^*_0,C^*_0) \leq E^*$, indicating $(u^*_0,v^*_0,C^*_0) \in S$. On the other hand, by \eqref{eqn:0} we have 
    \[
    H(u^*_{\lambda_k},v^*_{\lambda_k},C^*_{\lambda_k}) \leq \inf_{(u^*,v^*,C^*)\in S}H(u^*,v^*,C^*).
    \]
    Let $k\rightarrow\infty$ and the continuity of $H(u,v,C)$ yields $H(u^*_0,v^*_0,C^*_0) \leq \inf_{(u^*,v^*,C^*)\in S}H(u^*,v^*,C^*)$. Hence,  $(u^*_0,v^*_0,C^*_0)\in \arg\min\{ H(u,v,C): (u,v,C) \in S \}$. The strong convexity of $H$ and the convexity of $S$ imply the uniqueness of the solution to $\min\{ H(u,v,C): (u,v,C) \in S \}$. We therefore have shown that all the convergent subsequences of $(u^*_\lambda,v^*_\lambda,C^*_\lambda)$ converge to the unique solution of $\min\{ H(u,v,C): (u,v,C) \in S \}$ and thus the whole sequence converges to it.
\end{proof}

\subsection{Block Coordinate Descent for Inverse Bregman-Regularized OT}\label{sec:bcd}

Inspired by the work of \cite{ma2020learning} and the Sinkhorn algorithm \cite{cuturi2013sinkhorn}, we introduce the alternating minimization algorithm in \Cref{alg:AM} to solve \eqref{eqn:breg_iot_equiv_reg}. The Sinkhorn algorithm can be seen as an alternating minimization method applied to the dual problem of the entropy-regularized optimal transport problem.
\begin{algorithm}
\caption{Alternating Minimization Algorithm for \eqref{eqn:breg_iot_equiv_reg}}\label{alg:AM}
    \begin{algorithmic}
        \REQUIRE Initial point $u^0,v^0,C^0$, observed transport plan $\hat{X}$, regularization parameter $\gamma$ and $\lambda$.
        \REPEAT
        \STATE $u^k:= \arg\min_u E(u,v^{k-1},C^{k-1}) + \lambda R_1(u)$.
        \STATE $v^k:= \arg\min_v E(u^{k},v,C^{k-1}) + \lambda R_2(v)$.
        \STATE $C^k:= \arg\min_{C\in\mathcal{C}} E(u^k,v^{k},C) + \lambda R_3(C)$.
        \UNTIL{convergence}
        \STATE $C^* \leftarrow C^k$.
    \end{algorithmic}
\end{algorithm}

We present our main convergence result for \Cref{alg:AM}.
\begin{thm}
    Assume $R_1,R_2,R_3$ are continuously differentiable and strongly convex and suppose \Cref{asmp:breg} holds. Given any initial point $(u^0,v^0,C^0)$, let $\{(u^k,v^k,C^k)\}$ be generated by \Cref{alg:AM}. Then $\{E_\lambda(u^k,v^k,C^k)\}$ converges to $E_\lambda^*$ Q-linearly and $\{(u^k,v^k,C^k)\}$ converges to $(u^*_\lambda,v^*_\lambda,C^*_\lambda)$ R-linearly.
\end{thm}

\begin{proof}
The proof is almost identical to \cite[Proposition 3.4] {luo1993error} by the strong convexity of $E_\lambda(u,v,C)$ and we omit the details here.
\end{proof}

At present, we have introduced an Alternating Minimization (AM) algorithm (\Cref{alg:AM}), a form of Block Coordinate Descent (BCD), for minimizing the objective function $E_\lambda(z)$ and have established its R-linear convergence rate in theory.

While theoretically appealing, the practical implementation of the exact AM algorithm encounters a significant challenge: solving the block subproblems exactly at each iteration can be computationally expensive or even infeasible, as the block subproblems typically do not admit closed-form solutions.

To develop a more practical and computationally efficient approach, we propose an Inexact Block Coordinate Descent (IBCD) algorithm. This method preserves the beneficial cyclic update structure of the AM algorithm but replaces the exact minimization within each block with a computationally cheaper step that still guarantees sufficient progress. 

\subsubsection{Inexact BCD with Newton and Projected Gradient Steps}

\begin{algorithm}
\caption{Inexact Block Coordinate Descent Algorithm}\label{alg:IBCD}
\begin{algorithmic}[1]
\REQUIRE Initial point $z^0 = (u^0, v^0, C^0) \in \mathbb{R}^n \times \mathbb{R}^n \times \mathcal{C}$. Parameters $\gamma, \lambda > 0$. Line search parameters $\beta \in (0, 1)$, $c_1 \in (0, 1/2)$. Stopping criterion $\epsilon > 0$.
\STATE Set $k=0$.
\REPEAT
    \STATE \textbf{Update u-block:}
    \STATE 1a. Compute gradient $g_u^k = \nabla_u E_\lambda(u^k, v^k, C^k)$  and Hessian $H_u^k = \nabla^2_{uu} E_\lambda(u^k, v^k, C^k)$. 
    \STATE  1b. Compute Newton direction $d_u^k = -(H_u^k)^{-1} g_u^k$.
    \STATE  1c. Find step size $\alpha_k^u$: Start with $\alpha = 1$. \WHILE{$E_\lambda(u^k + \alpha d_u^k, v^k, C^k) > E_\lambda(u^k, v^k, C^k) + c_1 \alpha (g_u^k)^T d_u^k$}
        \STATE Set $\alpha \leftarrow \beta \alpha$.
    \ENDWHILE
    \STATE 1d.  Set $\alpha_k^u = \alpha$ and let $u^{k+1} = u^k + \alpha_k^u d_u^k$.

    \STATE \textbf{Update v-block similar to  u-block and obtain $v^{k+1}$:}

    \STATE \textbf{Update C-block:}
    \STATE 3a. Compute gradient $g_C^k = \nabla_C E_\lambda(u^{k+1}, v^{k+1}, C^k)$.
    \STATE 3b.  Let $d_C^k(s) = P_{\mathcal{C}}(C^k - s g_C^k) - C^k$ be the projected gradient direction for step size $s$.
    \STATE 3c. Find step size $\alpha_k^C$: Start with $\alpha = 1$. Let $C_{\text{trial}}(\alpha) = P_{\mathcal{C}}(C^k - \alpha  g_C^k)$.
    \WHILE{$E_\lambda(u^{k+1}, v^{k+1}, C_{\text{trial}}(\alpha)) > E_\lambda(u^{k+1}, v^{k+1}, C^k) + c_1 (g_C^k)^T (C_{\text{trial}}(\alpha) - C^k)$}
        \STATE Set $\alpha \leftarrow \beta \alpha$.
    \ENDWHILE
    \STATE 3d. Set $\alpha_k^C = \alpha$ and let $C^{k+1} = P_{\mathcal{C}}(C^k - \alpha_k^C g_C^k)$.

    \STATE $k \leftarrow k+1$.
\UNTIL{stopping criterion met}
\end{algorithmic}
\end{algorithm}


The detailed steps of our proposed inexact algorithm are presented in \Cref{alg:IBCD}. It follows the same cyclic update pattern as the exact AM method but utilizes specific inexact solvers for each block.

\Cref{alg:IBCD} preserves the overall structure of \Cref{alg:AM} (the exact AM method) while modifying the core block updates. Instead of solving for the exact minimum within each block, the algorithm performs a single, computationally bounded step designed to ensure progress. For the $u$ and $v$ blocks, a damped Newton step is employed, incorporating second-order information about the subproblem. To guarantee sufficient decrease in the overall objective function $E_\lambda$, the Armijo line search is used. For the $C$ block, a projected gradient step is performed, efficiently enforcing the constraint $C \in \mathcal{C}$ via the projection $P_{\mathcal{C}}$. 

The key feature of \Cref{alg:IBCD} is that each block update requires only a fixed, bounded amount of computation, avoiding indefinite inner iterations. This ensures that the overall cost per iteration is predictable and often significantly lower than that of the exact AM method. The critical question, which is addressed in the next section, is whether this practical inexact approach retains the guaranteed linear convergence rate of its exact counterpart.

\subsubsection{Convergence Analysis of Inexact Block Coordinate Descent Algorithm}

We establish the convergence properties of \Cref{alg:IBCD} under \Cref{asmp:IBCD}. Note that \Cref{asmp:IBCD} implies the compactness of the sublevel set of $E_\lambda(u,v,C)$ and the Lipschitz continuity of $\nabla E_\lambda$ on any bounded set. 

\begin{asmp}\label{asmp:IBCD}
\ \newline \vspace{-5mm}
    \begin{itemize}
        \item[(A1)] $\psi$ is of Legendre type  and $C^2$ on $\operatorname{int}(\operatorname{dom}\psi)$.
        \item[(A2)] The observed transport plan $\hat{X}_{ij}$ and the generated iterate $(u^k, v^k, C^k)$ satisfy $\hat{X}_{ij} \text{ and } (u_i^k + v_j^k - C_{ij}^k)/\gamma \in \operatorname{int}(\operatorname{dom}\psi)$ for all $i,j$.
        \item[(A3)] $R_1, R_2: \mathbb{R}^n \to \mathbb{R}$ and $R_3: \mathbb{R}^{n \times n} \to \mathbb{R}$ are strongly convex and continuously differentiable functions with their strong convexity constants be $\sigma_1, \sigma_2, \sigma_3 > 0$.
        \item[(A4)] $\mathcal{C} \subseteq \mathbb{R}^{n \times n}$ is a nonempty, closed, and convex set.
    \end{itemize}
\end{asmp}

Under these conditions, \Cref{alg:IBCD} not only converges to the unique minimizer of $E_\lambda$ but does so at a linear rate. We summarize this result in the following theorem.

\begin{thm}\label{thm:IBCD_convergence}
    Let \Cref{asmp:IBCD} hold. Let the regularizer parameter $\gamma,\lambda >0$ and the line search parameters satisfy $\beta \in (0, 1)$ and $c_1 \in (0, 1/2)$. Then the sequence $\{z^k = (u^k, v^k, C^k)\}$ generated by \Cref{alg:IBCD} satisfies that $\{E_\lambda(u^k,v^k,C^k)\}$ converges to $E_\lambda^*$ Q-linearly and $\{(u^k,v^k,C^k)\}$ converges to $(u^*_\lambda,v^*_\lambda,C^*_\lambda)$ R-linearly.
\end{thm}

To prove \Cref{thm:IBCD_convergence}, we begin by establishing the common notations and key properties derived from \Cref{asmp:IBCD}. These will serve as the foundation for the subsequent analysis.

Denote $z = (u,v,C)$ and $f(z) = E_\lambda(z)$. By (A3) and (A4) of \Cref{asmp:IBCD}, $f(z)$ is strongly convex with modulus $\sigma = \lambda \min(\sigma_1, \sigma_2, \sigma_3) > 0$ and have a unique minimizer $z^* \in \mathbb{R}^n \times \mathbb{R}^n \times \mathcal{C}$. As a corollary of \Cref{asmp:IBCD}, the level set $\mathcal{S}_0 = \{z \in \mathbb{R}^n \times \mathbb{R}^n \times \mathcal{C} \mid f(z) \le f(z^0)\}$ is compact. Denote $z^{k,u} = (u^{k+1},v^k,C^k)$ and $z^{k,v} = (u^{k+1},v^{k+1},C^k)$ and define $D_u^k := f(z^k)-f(z^{k,u})$, $D_v^k := f(z^{k,u})-f(z^{k,v})$, and $D_C^k := f(z^{k,v})-f(z^{k+1})$. Since the algorithm ensures $f(z^{k+1}) \le f(z^{k,v}) \leq f(z^{k,u}) \le f(z^k)$ due to the Armijo condition and descent directions $d_u^k,d_v^k,d_C^k$, all iterates $\{z^k\}$ remain in the bounded set $\mathcal{S}_0$. As a result of \Cref{asmp:IBCD}, $\nabla f$ is Lipschitz continuous on $\mathcal{S}_0$ with some constant $L < \infty$. This also implies that the block Hessians $\nabla^2_{uu}f, \nabla^2_{vv}f, \nabla^2_{CC}f$ have eigenvalues bounded above by $L_u, L_v, L_C$ respectively (where $L_u, L_v, L_C \le L$) on $\mathcal{S}_0$. So the block subproblems $f_u^k(u) := f(u, v^k, C^k)$, $f_v^k(v) := f(u^{k+1}, v, C^k)$, and $f_C^k(C) := f(u^{k+1}, v^{k+1}, C)$ (restricted to $C \in \mathcal{C}$) are strongly convex with moduli $\sigma_u = \lambda \sigma_1$, $\sigma_v = \lambda \sigma_2$, and $\sigma_C = \lambda \sigma_3$, respectively. They are also smooth with Lipschitz gradients bounded by $L_u, L_v, L_C$ on the relevant domains defined by $\mathcal{S}_0$.

The proof proceeds in two main steps, formalized through the following lemmas: Lemma~\ref{lem:sufficient_decrease}, which establishes sufficient decrease at each step, and Lemma~\ref{lem:bounded_steps}, which ensures that the Armijo step sizes are uniformly bounded away from zero.

\begin{lem}[Sufficient Decrease per Block]
\label{lem:sufficient_decrease}
Let \Cref{asmp:IBCD} hold. There exists a constant $w_1 >0$, such that
\begin{equation}
    f(z^k) - f(z^{k+1}) \geq w_1 \|z^{k+1} - z^k\|^2.
\end{equation}
\end{lem}

\begin{proof}
Let $\Delta u = u^{k+1} - u^k = \alpha_k^u d_u^k$. From the Armijo condition: 
$$D_u^k \ge  -c_1 \alpha_k^u \langle g_u^k, d_u^k \rangle = c_1 \alpha_k^u \langle g_u^k, (H_u^k)^{-1} g_u^k \rangle.$$ 
Since $H_u^k \succeq \sigma_u I$,  we have 
$$\langle g_u^k, (H_u^k)^{-1} g_u^k \rangle \ge \sigma_u \|(H_u^k)^{-1} g_u^k\|^2 = \sigma_u \|d_u^k\|^2.$$ 
As a result, $$D_u^k \ge c_1 \alpha_k^u \sigma_u \|d_u^k\|^2 \ge  \frac{c_1 \sigma_u}{\alpha_k^u} \|\Delta u\|^2,$$ since $\Delta u = \alpha_k^u d_u^k$. Note that the step size is always $\alpha_k^u \le  1$. Therefore, $D_u^k \ge  c_1 \sigma_u \|u^{k+1} - u^k\|^2$. Let $\tilde{\omega}_u = c_1 \sigma_u > 0$. We have 
$$D_u^k \ge \tilde{\omega}_u \|u^{k+1} - u^k\|^2.$$

Similarly let $\tilde{\omega}_v = c_1 \sigma_v > 0$. We have $$D_v^k \ge \tilde{\omega}_v \|v^{k+1} - v^k\|^2.$$

Let $\Delta C = C^{k+1} - C^k$. The Armijo condition implies $$D_C^k = f_C^k(C^k) - f_C^k(C^{k+1}) \ge -c_1 \langle g_C, \Delta C \rangle.$$
  By the projection property, we have: 
    $$\langle (C^k - \alpha_k^C g_C) - C^{k+1}, C^k - C^{k+1} \rangle \le 0.$$
So we have $\langle g_C, \Delta C \rangle \le -\frac{1}{\alpha_k^C} \|\Delta C\|^2$. (This requires $\alpha_k^C > 0$, which holds if $\Delta C \neq 0$).  Substitute into the Armijo bound:
    $$D_C^k \ge -c_1 \langle g_C, \Delta C \rangle \ge  \frac{c_1}{\alpha_k^C} \|\Delta C\|^2 \ge c_1\|\Delta C\|^2.$$
Therefore, we have:
$$
\begin{aligned}
    f(z^k) - f(z^{k+1}) &= D_u^k + D_v^k + D_C^k \\
    &\ge \tilde{\omega}_u \|u^{k+1}-u^k\|^2 + \tilde{\omega}_v \|v^{k+1}-v^k\|^2 + \tilde{\omega}_C \|C^{k+1}-C^k\|^2 \\
    &\ge w_1 (\|u^{k+1}-u^k\|^2 + \|v^{k+1}-v^k\|^2 + \|C^{k+1}-C^k\|^2) \\
    &=w_1 \|z^{k+1} - z^k\|^2,
\end{aligned}
$$
where $w_1 = \min(\tilde{\omega}_u, \tilde{\omega}_v, \tilde{\omega}_C) = c_1 \min(\sigma_u, \sigma_v, 1) > 0$. 
\end{proof}

\begin{lem}[Uniformly Bounded Step Sizes]
\label{lem:bounded_steps}
Let \Cref{asmp:IBCD} hold. The Armijo stepsize $\alpha_k^u,\alpha_k^v,\alpha_k^C$ have a uniform strictly positive lower bound $\bar{\alpha}_{min} > 0$.
\end{lem}
\begin{proof}
Without loss of generality, we only show that such a lower exists for $\alpha_k^C$. The line search starts with $\alpha_{init}=1$. If $\alpha=1$ works, $\alpha_k^C=1$. If not, it reduces by $\beta$. The first step $\alpha'$ that fails the Armijo condition is $\alpha_k^C / \beta$. This step failed because:
    $$f_C^k(C^k + d(\alpha')) > f_C^k(C^k) + c_1 \langle g_C, d(\alpha') \rangle,$$ where $d(\alpha') = P_{\mathcal{C}}(C^k - \alpha' g_C) - C^k$. Using $L_C$-smoothness: $$f_C^k(C^k + d(\alpha')) \le f_C^k(C^k) + \langle g_C, d(\alpha') \rangle + \frac{L_C}{2} \|d(\alpha')\|^2.$$
     Combining gives: 
     $$(c_1 - 1)\langle g_C, d(\alpha') \rangle < \frac{L_C}{2} \|d(\alpha')\|^2.$$
    From the projection property derived above: $$\langle g_C, d(\alpha') \rangle \le -(1/\alpha') \|d(\alpha')\|^2.$$
    Substitute: $$\frac{1-c_1}{\alpha'} \|d(\alpha')\|^2 \le (c_1 - 1)\langle g_C, d(\alpha') \rangle < \frac{L_C}{2} \|d(\alpha')\|^2.$$
    Assuming $d(\alpha') \neq 0$ (i.e., $C^k$ is not optimal for the subproblem), we can divide by $\|d(\alpha')\|^2$ and obtain that $\frac{1-c_1}{\alpha'} < \frac{L_C}{2} \implies \alpha' > \frac{2(1-c_1)}{L_C}.$ Since the failed step was $\alpha' = \alpha_k^C / \beta$, we have $\alpha_k^C / \beta > \frac{2(1-c_1)}{L_C}$. Thus, $\alpha_k^C > \frac{2\beta(1-c_1)}{L_C}$. Combining with the case where $\alpha_k^C = \alpha_{init} = 1$, we have a uniform lower bound: $\bar{\alpha}_{C, min} =: \min(1, \frac{2\beta(1-c_1)}{L_C}) > 0$. Denote $\bar{\alpha}_{min} > 0$ the common lower bound of the Armijo stepsize $\alpha_k^u,\alpha_k^v,\alpha_k^C$.
\end{proof}

Finally, we return to the proof of our main theorem.
\begin{proof}[Proof of \Cref{thm:IBCD_convergence}]
     
Define$ \mathcal{Z} = \mathbb{R}^n \times \mathbb{R}^n \times \mathcal{C}$. Now rewrite our iteration into
$$z^{k+1} = [z^k - \alpha_k^C \nabla f(z^k) + e^k]_{\mathcal{Z}},$$
where $e^k = (e_u^k, e_v^k, e_C^k)$ is defined by 
$$
\begin{cases}
    e_u^k = (\alpha_k^C I - \alpha_k^u H_u^{-1}) g_u(z^k) \\
    e_v^k = \alpha_k^C (g_v(z^k) - g_v^{k,u}) + (\alpha_k^C I - \alpha_k^v H_v^{-1}) g_v^{k,u} \\
    e_C^k = \alpha_k^C (g_C(z^k) - g_C^{k,v})
\end{cases}.
$$
We wish to bound $\|e^k\|^2 = \|e_u^k\|^2 + \|e_v^k\|^2 + \|e_C^k\|^2$ by $M_{max} \|z^k - z^{k+1}\|^2$ for some constant $M_{max}$. To bound $e_u^k$, we find  $\|e_u^k\| \le \|\alpha_k^C I - \alpha_k^u H_u^{-1}\| \|g_u^k\|$. Since $\alpha_k^C, \alpha_k^u  \le 1$ and $H_u^k \succeq \sigma_u I$, the operator norm $\|\alpha_k^C I - \alpha_k^u H_u^{-1}\|$ is uniformly bounded by  $1+1/\sigma_u$. Combining \Cref{lem:bounded_steps}, we have
    $$\|e_u^k\| \le \frac{1+\sigma_u}{\sigma_u} \|g_u^k\| \le\frac{(1+\sigma_u)L_u}{\sigma_u\bar{\alpha}_{min}}  \|\Delta u\|.$$
As for $e_v^k$, we have:
$$
\begin{aligned}
    \|e_v^k\| &\le \alpha_k^C \|g_v(z^k) - g_v^{k,u}\| + \|\alpha_k^C I - \alpha_k^v H_v^{-1}\| \|g_v^{k,u}\| \\
    & \le  L \|\Delta u\| + \frac{1+\sigma_v}{\sigma_v} \|g_v^{k,u}\| \\
    & \le  L \|\Delta u\| + \frac{(1+\sigma_v)L_v}{\sigma_v\bar{\alpha}_{min}} \|\Delta v\| \\
\end{aligned}
$$
Finally for $e_C^k$:
    $$\|e_C^k\| = \|\alpha_k^C (g_C(z^k) - g_C^{k,v})\| \le L\|\Delta u\| + L\|\Delta v\|.$$
Simple calculations show that there exists a constant $M_{max}$ such that 
    $$\|e^k\|^2 \le M_{max} (\|\Delta u\|^2 + \|\Delta v\|^2 + \|\Delta C\|^2) = M_{max} \|z^k - z^{k+1}\|^2.$$

By \cite[Theorem 3.1]{luo1993error} and \Cref{lem:sufficient_decrease}, we obtain  Q-linear convergence for $\{f(z^k)\}$ and R-linear convergence for $\{z^k\}$.

\end{proof}

\section{Experiments}\label{sec:experiment}
In this section, we present several numerical examples of IOT using the proposed algorithms. For the IOT problem formulated in \eqref{eqn:breg_iot_equiv_reg}, we specifically choose quadratic regularizers: $R_1(u) = \frac{1}{2}\|u\|_2^2$, $R_2(v) = \frac{1}{2}\|v\|_2^2$, and $R_3(C) = \frac{1}{2}\|C\|_F^2$. This choice significantly simplifies the subproblems within \Cref{alg:IBCD}. With these quadratic regularizers, the Hessian matrices for the $u$, $v$, and $C$ subproblems become diagonal. Consequently, the Newton step for each block can be computed efficiently through element-wise division by the diagonal Hessian entries, rather than solving a linear system, which substantially accelerates each iteration.

Moreover, for the $C$-block update, while \Cref{alg:IBCD} generally describes a projected gradient descent step, our practical implementation for the experiments employs a projected Newton step. Specifically, the Newton direction is computed as $d_C = - (\nabla^2_{CC} E_\lambda)^{-1} \nabla_C E_\lambda$, which simplifies to element-wise division due to the diagonal structure of the Hessian. The $C$-block is then updated via $C \leftarrow \text{Proj}_{\mathcal{C}}\big(C_\text{old} + \alpha_C d_C\big)$, where $\alpha_C$ is determined using a line search. Empirically, the Newton-based update for $C$ has been observed to yield faster convergence and better or comparable solution quality for these regularizers. The detailed iterative scheme employed in this section is provided in \Cref{alg:practical_bcd}.
\begin{algorithm}[h]
\caption{Practical BCD for IOT with Quadratic Regularizers}
\label{alg:practical_bcd}
\begin{algorithmic}[1]
\REQUIRE Initial $u^{(0)}, v^{(0)}, C^{(0)}$, observed $\hat{X}$, OT regularizer $\gamma$, BCD regularizer $\lambda$.
\FOR{$k=0, 1, 2, \dots$ until convergence}
    \STATE Set Newton direction $d_u = -\frac{\psi^\prime(\frac{u\oplus v^{k-1}-C^{k-1}}{\gamma})\mathbf{1} - \mu + \lambda u}{\frac{1}{\gamma}\psi^{\prime\prime}(\frac{u\oplus v^{k-1}-C^{k-1}}{\gamma})\mathbf{1}  + \lambda }$.
    \STATE Find step size $\alpha_u$ via line search.
    \STATE $u^{(k+1)} = u^{(k)} + \alpha_u d_u$.
    \STATE Set Newton direction $d_v = - \frac{\psi^\prime(\frac{u^k\oplus v-C^{k-1}}{\gamma})^\top\mathbf{1} - \nu + \lambda v}{\frac{1}{\gamma}\psi^{\prime\prime}(\frac{u^k\oplus v-C^{k-1}}{\gamma})^\top\mathbf{1}  + \lambda }$.
    \STATE Find step size $\alpha_v$ via line search.
    \STATE $v^{(k+1)} = v^{(k)} + \alpha_v d_v$.
    \STATE Set Newton direction $d_C = - \frac{\hat{X} - \psi^\prime(\frac{u^k\oplus v^{k}-C }{\gamma})  + \lambda C}{\frac{1}{\gamma}\psi^{\prime\prime}(\frac{u^k\oplus v^{k}-C}{\gamma})  + \lambda }$ .
    \STATE Find step size $\alpha_C$ via line search.
    \STATE $C^{(k+1)} = \text{Proj}_{\mathcal{C}}(C^{(k)} + \alpha_C d_C)$ .
\ENDFOR
\RETURN $u^{(k+1)}, v^{(k+1)}, C^{(k+1)}$.
\end{algorithmic}
\end{algorithm}

\begin{rmk}[Computational Complexity]
\label{rem:complexity}
The per-iteration computational complexity of Algorithm \ref{alg:practical_bcd} is dominated by the matrix operations required to compute the gradients and Hessians for the $u, v$, and $C$ subproblems. Specifically, evaluating the terms involving the Bregman generator $\psi$ (e.g., $\psi'((u \oplus v - C)/\gamma)$) requires element-wise operations on $n \times n$ matrices, leading to a complexity of $\mathcal{O}(n^2)$. When choosing $\mathcal{C} = S_h$, the projection step $\operatorname{Proj}_{\mathcal{C}}$ involves element-wise thresholding and symmetrization, which also costs $\mathcal{O}(n^2)$.  Crucially, thanks to the choice of quadratic regularizers $R_1, R_2, R_3$, the subproblems admit a diagonal Hessian structure. This allows the Newton directions $d_u, d_v, d_C$ to be computed via simple element-wise division, avoiding the computationally expensive solution of linear systems (which would typically scale as $\mathcal{O}(n^3)$ or higher). Consequently, the overall per-iteration complexity of our algorithm is $\mathcal{O}(n^2)$, scaling comparably to the standard Sinkhorn algorithm used for forward optimal transport.
\end{rmk}

Specific parameters for each experiment are detailed within their respective subsections. For \Cref{alg:practical_bcd}, convergence was typically assessed based on the value of the relative KKT residuals falling below a threshold of $10^{-6}$, with a maximum iteration limit of 100. More precisely, let $(u^{(k)}, v^{(k)}, C^{(k)})$ be the iterates at the $k$-th iteration. The absolute KKT residual $\mathcal{R}_k$ is defined to measure the violation of the first-order optimality conditions for the problem in \ref{eqn:breg_iot_equiv_reg}, which is computed as:
$$
\mathcal{R}_k := \sqrt{ \|\nabla_u E_\lambda(u^{(k)}, v^{(k)}, C^{(k)})\|^2_2 + \|\nabla_v E_\lambda(u^{(k)}, v^{(k)}, C^{(k)})\|^2_2 + \|C^{(k)} - \text{Proj}_{\mathcal{C}}(C^{(k)} - \nabla_C E_\lambda(u^{(k)}, v^{(k)}, C^{(k)}))\|_F^2 }
$$
The relative KKT residual is thus evaluated by $\mathcal{R}_k / \mathcal{R}_0$. For the choice of regularizer $\phi$, We mainly focus on the following choice of $\phi$ in \Cref{tab:bregman_derivative}. 
\begin{table}[]
    \centering
    \begin{tabular}{llll}
\hline$\phi(x)$ & $\phi^{\prime}(x)$ & $\psi^{\prime}(\theta)$ & $\psi^{\prime \prime}(\theta)$ \\
\hline 
$x \log x-x+1$ & $\log x$ & $\exp \theta$ & $\exp \theta$ \\
$x-\log x-1$ & $1-x^{-1}$ & $(1-\theta)^{-1}$ & $(1-\theta)^{-2}$ \\
$x \log x+(1-x) \log (1-x)$ & $\log \frac{x}{1-x}$ & $\frac{\exp \theta}{(1+\exp \theta)}$ & $\frac{\exp \theta}{(1+\exp \theta)^2}$ \\
$\frac{1}{\beta(\beta-1)}(x^\beta - \beta x+\beta-1)$ & $\frac{1}{\beta-1}(x^{\beta-1} -1)$ & $((\beta-1) \theta+1)^{\frac{1}{\beta-1}}$ & $((\beta-1) \theta+1)^{\frac{1}{\beta-1}-1}$ \\
\hline
\end{tabular}
    \caption{Bregman regularizers and their derivatives}
    \label{tab:bregman_derivative}
\end{table}

All experiments, unless otherwise specified, were conducted in MATLAB R2022a on a laptop equipped with an 11th Gen Intel(R) Core(TM) i7-11800H @ 2.30GHz CPU. The real-world marriage matching experiment was implemented in Python 3.13.2.

\subsection{Experiments on randomly generated marginals}
In this section, we fix a groud truth $C\in\mathbb{R}^{n\times n}$ with $C_{ij} = |\frac{i-j}{n}|^2$ and randomly generate the marginals $\mu$ and $\nu$. Given $\mu,\nu$ and $C$, we employ Alternate scaling algorithm in \cite{dessein2018regularized} to compute a transport plan $X$ as our observed transport plan. Then based on $X$ we use \Cref{alg:practical_bcd} to recover a cost matrix $\hat{C}$. Using this recovered cost matrix $\hat{C}$, we again apply Alternate scaling algorithm in \cite{dessein2018regularized} to $(\hat{C},\mu,\nu)$ and obtain a transport plan $\hat{X}$. For each regularizer, we repeat our experiments for 10 times and compute the average relative error of $C$ : $\frac{\|C-\hat{C}\|_F}{\|C\|_F}$ denoted by \textbf{C err}, the average relative error of $X$ : $\frac{\|X-\hat{X}\|_F}{\|X\|_F}$ denoted by \textbf{X err}, and the average running time \textbf{t}. We present our results in \Cref{tab:res_sh} and \Cref{tab:res_ed}. The reason why \textbf{C err} for the case of Burg entropy $\phi_1(x) = x-\log x-1$ and $\beta$-potential $\phi_2(x) = \frac{1}{\beta(\beta-1)}\left(x^\beta-\beta x+\beta-1\right)$ is large is probably because $\operatorname{int}(\operatorname{dom}(\psi_i))$ is bounded above (by $1$ for $\psi_1$, by $1/(1-\beta)$ for $\psi_2$). The objective function and its hessian require us to compute $\psi_i(z)$ and $\psi_i^{\prime\prime}(z)$ with $z = \frac{u_i+v_j-C_{ij}}{\gamma}$. But the algorithm itself doesn't impose harsh constraint so that $z\in\operatorname{int}(\operatorname{dom}(\psi_i))$. Also as $z$ tends to the upper bound, both $\psi_i(z)$ and $\psi_i^{\prime\prime}(z)$ tend to $\infty$. This results in serious numerical instability issue.

\begin{table}
				\centering
				\begin{tabular}{cccc}
					\hline
					$n=10$  & \textbf{C err} & \textbf{X err} & \textbf{t} \\
					\hline
					Boltzmann-Shannon entropy  & 0.021 & $9.18\times 10^{-4}$ & 0.0044 \\
					Burg entropy  & 0.93& 0.0027& 0.015\\
					Fermi-Dirac entropy  & 0.033 & 0.0025 & 0.0088\\
					$\beta$-potential  & 0.85 & 0.030 & 0.21\\
					\hline
					$n=50$ & \textbf{C err} & \textbf{X err} & \textbf{t} \\
					\hline
					Boltzmann-Shannon entropy  & 0.42 & 0.0474 & 0.029 \\
					Burg entropy  & 1.05& $1.86\times10^{-4}$& 0.076\\
					Fermi-Dirac entropy  & 0.43 & 0.045 & 0.049\\
					$\beta$-potential  & 0.86 & 0.023 & 0.068\\
					\hline
					$n=100$ & \textbf{C err} & \textbf{X err} & \textbf{t} \\
					\hline
					Boltzmann-Shannon entropy  & 0.69 & 0.11 & 0.058 \\
					Burg entropy  & 1.05 & $4.30\times10^{-5}$& 0.011\\
					Fermi-Dirac entropy  & 0.59 & 0.086 & 0.10\\
					$\beta$-potential  & 0.86 & 0.019 & 0.12\\
					\hline
				\end{tabular}
				\caption{This table presents the experimental results evaluated on randomly generated marginals when $\mathcal{C} = S_h$. In these experiments, the BCD regularization parameter is fixed at $\lambda = 10^{-8}$, and we set $\beta = 0.8$ for the $\beta$-potential.}
				\label{tab:res_sh}
\end{table}

\begin{table}
    \centering
    \begin{tabular}{cccc}
    \hline
       $n=10$& \textbf{C err} & \textbf{X err} & \textbf{t} \\
         \hline
       Boltzmann-Shannon entropy  & 0.018& $7.51\times 10^{-4}$& 0.027\\
       Burg entropy  & 0.94& 0.0029& 0.23\\
       Fermi-Dirac entropy  & 0.011& $5.13\times 10^{-4}$& 0.055\\
       $\beta$-potential  & 0.86& 0.030& 0.64\\
       \hline
        $n=50$& \textbf{C err} & \textbf{X err} & \textbf{t} \\
         \hline
       Boltzmann-Shannon entropy  & 0.44& 0.0481& 0.13\\
       Burg entropy  & 1.16& $1.56\times10^{-4}$& 6.05\\
       Fermi-Dirac entropy  & 0.41& 0.044& 0.18\\
       $\beta$-potential  & 0.86& 0.023& 0.20\\
       \hline
        $n=100$& \textbf{C err} & \textbf{X err} & \textbf{t} \\
         \hline
       Boltzmann-Shannon entropy  & 0.67& 0.10& 0.37\\
       Burg entropy  & 1.06& $4.72\times10^{-5}$& 4.81\\
       Fermi-Dirac entropy  & 0.59& 0.088& 0.67\\
       $\beta$-potential  & 0.86& 0.019& 0.38\\
       \hline
    \end{tabular}
    \caption{This table summarizes the recovery performance on randomly generated marginals when $\mathcal{C} = \operatorname{ED}$. The parameters are uniformly set to $\lambda = 10^{-8}$, with $\beta = 0.8$ specifically chosen for the $\beta$-potential case.}
    \label{tab:res_ed}
\end{table}

\subsection{Verification of Stability Bound}

This experiment aims to numerically verify the stability bound derived in \Cref{prop:stab} for the case where $\phi'_0 = -\infty$.  The cost matrices are restricted to the set $S_h$. We set the problem dimension $n=10$. For $\gamma\in(0.01,10)$ , we repeat each test 50 times. In each test, we generate random marginals $\mu, \nu $. A ground truth cost matrix $C \in S_h$ is randomly generated. The corresponding transport plan $X = F(C)$ is computed using the  Alternate scaling algorithm in \cite{dessein2018regularized} . A perturbed cost matrix $\hat{C} = \text{Proj}_{S_h}(C + \delta N)$ is created by adding scaled random noise $N$ with entries drawn from a standard normal distribution, projected onto $S_h$ to $C$, where the noise level $\delta = 0.01  / ||N||_{\infty}$ controls the perturbation magnitude. The corresponding plan $\hat{X} = F(\hat{C})$ is then computed. We calculate the Left-Hand Side (LHS) $||\hat{C} - C||_{\infty}$ and the Right-Hand Side (RHS) $2\gamma ||\nabla\phi(\hat{X}) -\nabla\phi(X)||_{\infty}$, taking care to handle potential near-zero entries in $X$ by adding a small  $\epsilon=10^{-20}$ before taking the logarithm.  Across all tested $\gamma$ values and trials, the inequality LHS $\leq$ RHS was consistently satisfied within numerical precision tolerance $10^{-9}$, providing strong numerical validation for \Cref{prop:stab} under the specified conditions. We also analyzed the ratio RHS/LHS. The minimum observed ratio was consistently greater than or equal to 1, as expected. The average ratio tended to increase with $\gamma$, suggesting the bound becomes less tight as the OT regularization increases. Representative results illustrating the pass rate and ratio behavior versus $\gamma$ are shown in \Cref{fig:stability}.

\begin{figure}[htbp]
    \centering
    \subfloat[Boltzmann-Shannon entropy]{\includegraphics[width=0.4\textwidth]{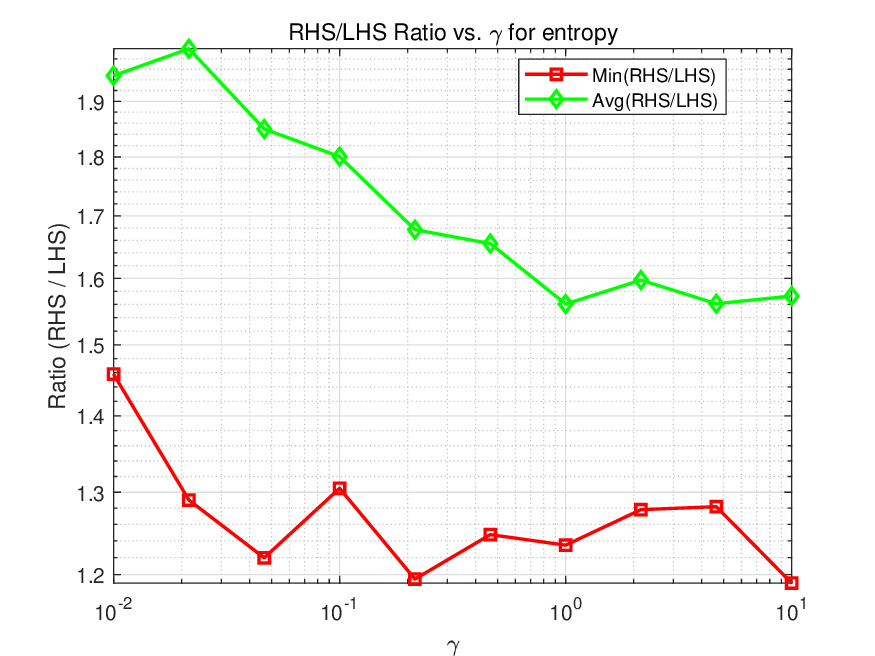}}
    \hfill
    \subfloat[Burg entropy]{\includegraphics[width=0.4\textwidth]{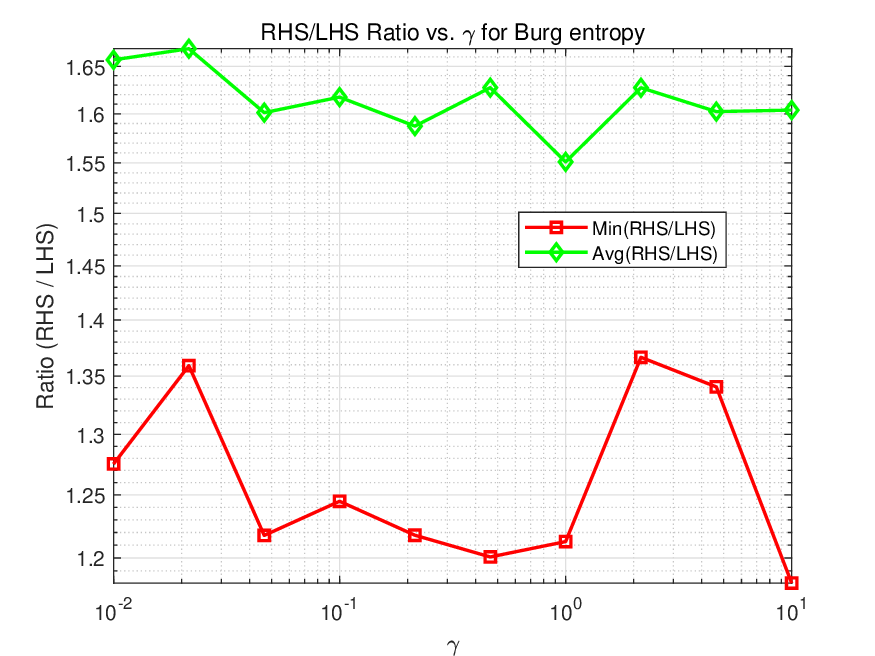}}
    
    \vspace{0.2cm} 
    
    \subfloat[Fermi-Dirac entropy]{\includegraphics[width=0.4\textwidth]{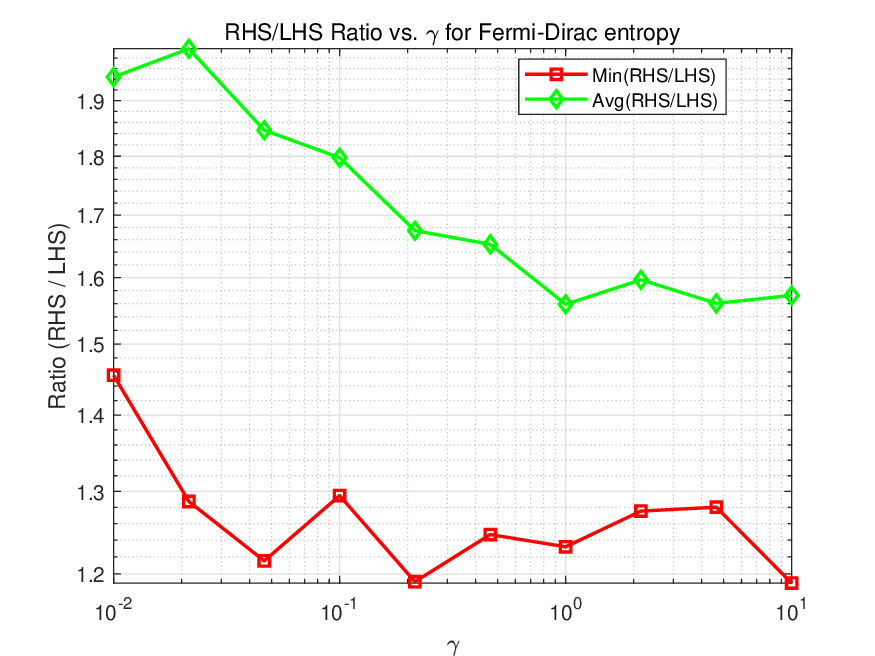}}
    \hfill
    \subfloat[$\beta$-potential]{\includegraphics[width=0.4\textwidth]{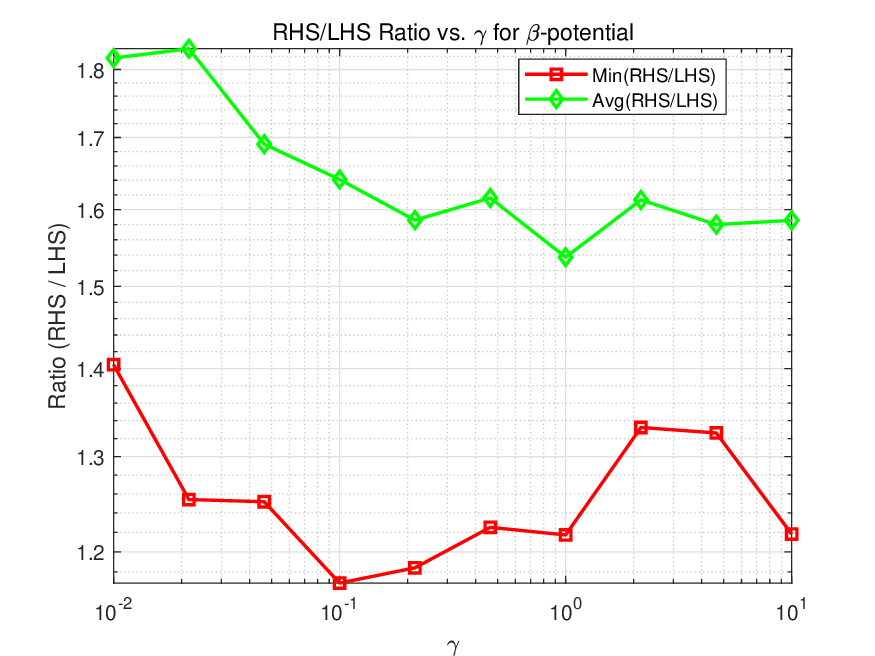}}
    \caption{This figure illustrates the numerical verification of the theoretical stability bounds across four different types of Bregman regularizers. For the $\beta$-potential regularizer, the parameter is set to $\beta=0.5$.}
    \label{fig:stability}
\end{figure}

\subsection{Effect of Regularization Parameter $\lambda$}

This experiment investigates the impact of the regularization parameter $\lambda$ from the penalized objective function in \eqref{eqn:breg_iot_equiv_reg} on the solution quality when using \Cref{alg:IBCD}  to solve the inverse problem.. The cost matrix search space is $S_h$. We test the problem with dimension $n=10$ and OT regularization $\gamma=0.1$ and generate a ground truth $C \in S_h$ and corresponding marginals $\mu, \nu$. The "observed" transport plan $X = F(C)$ is computed using the  Alternate scaling algorithm in \cite{dessein2018regularized} . We then employ \Cref{alg:IBCD}  to recover an estimated cost matrix $\hat{C}(\lambda)$ from $X$ for a range of $\lambda$ values. For each recovered $\hat{C}(\lambda)$, we also compute the corresponding transport plan $\hat{X}(\lambda) = F(\hat{C}(\lambda))$ using the  Alternate scaling algorithm in \cite{dessein2018regularized} . We evaluate the recovery quality using the relative error for the cost matrix, $||\hat{C}(\lambda) - C||_F / ||C||_F$, and the relative error for the transport plan, $||\hat{X}(\lambda) - X||_F / ||X||_F$. The results are displayed in \Cref{fig:lambda1} and \Cref{fig:lambda2}, which demonstrates the critical sensitivity of the IOT recovery to the regularization parameter $\lambda$. Our experimental results indicate that to achieve minimal relative error in both the recovered cost matrix $\hat{C}$  and the subsequently derived transport plan $\hat{X}(\lambda)$, $\lambda$ must be kept exceptionally small, typically in the order of $10^{-10}$ or lower. Even a marginal increase in $\lambda$ beyond this very low threshold, for instance to $10^{-8}$ or higher, results in a dramatically sharp increase in the cost matrix recovery error, underscoring the necessity of minimal regularization for faithful cost structure inference in this setting.

\begin{figure}[htbp]
    \centering
    \includegraphics[width=0.85\textwidth]{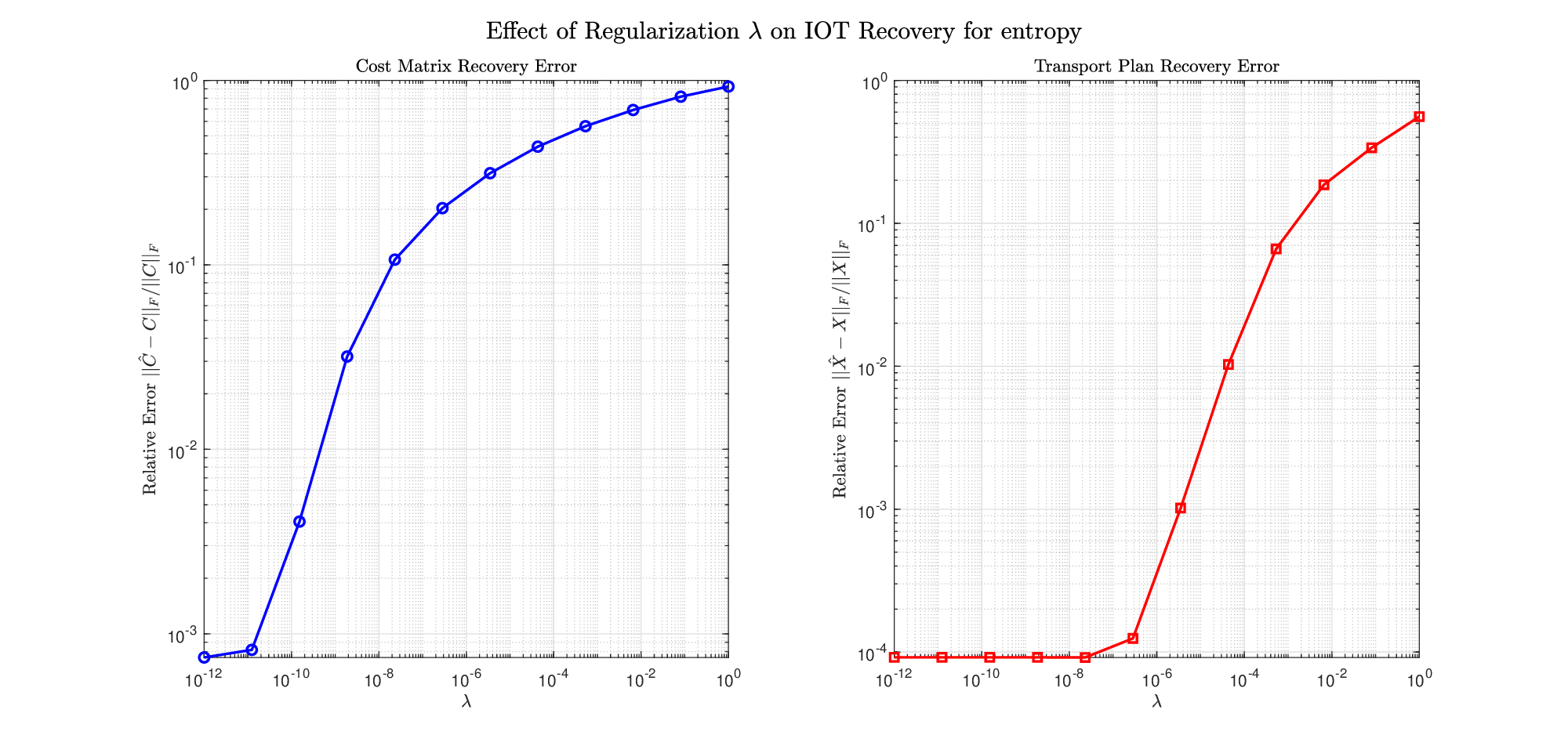}
    \includegraphics[width=0.85\textwidth]{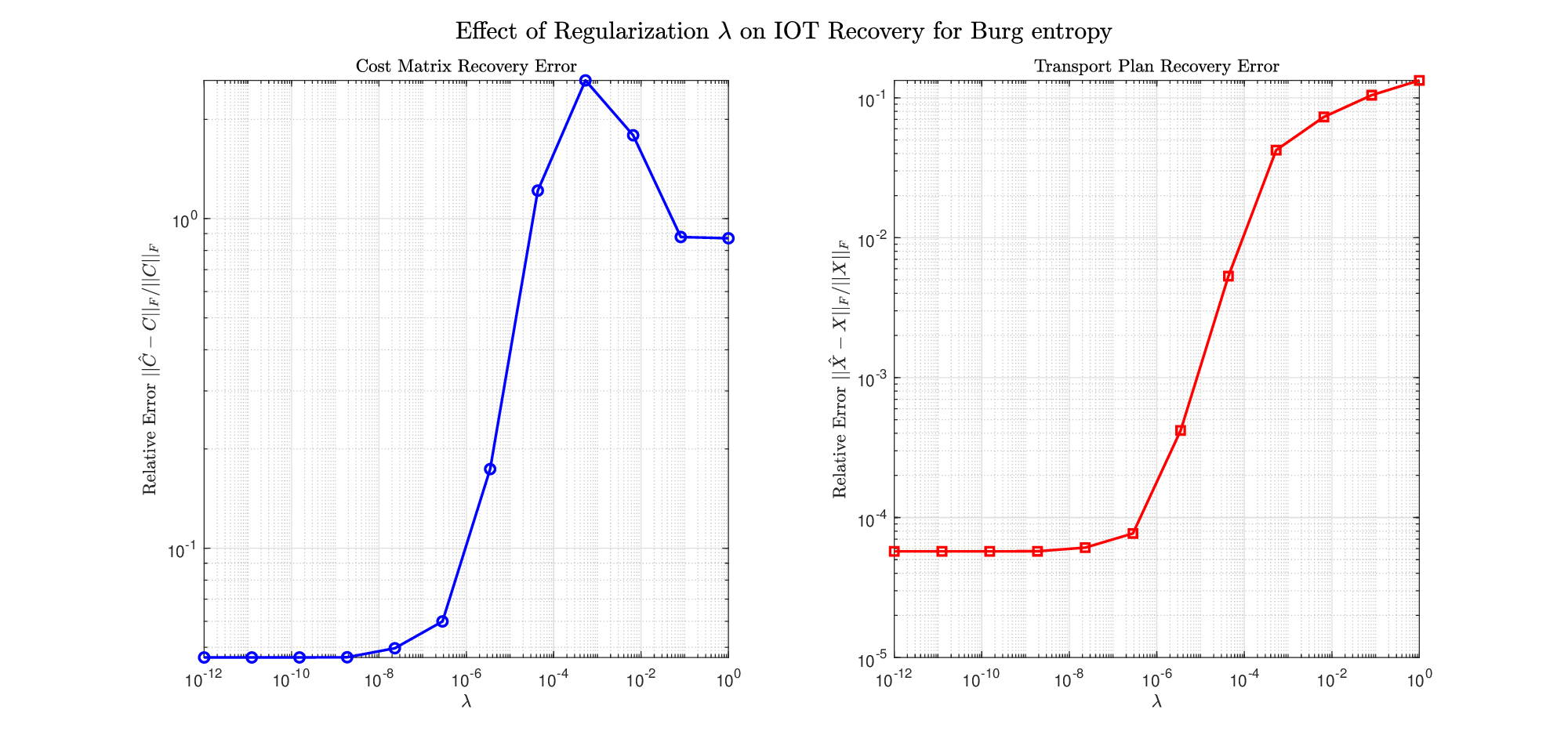}
    \caption{This figure demonstrates the effect of the regularization parameter $\lambda$ on the recovery performance when using the Boltzmann-Shannon entropy and Burg entropy.}
    \label{fig:lambda1}
\end{figure}

\begin{figure}[htbp]
    \centering
    \includegraphics[width=0.85\linewidth]{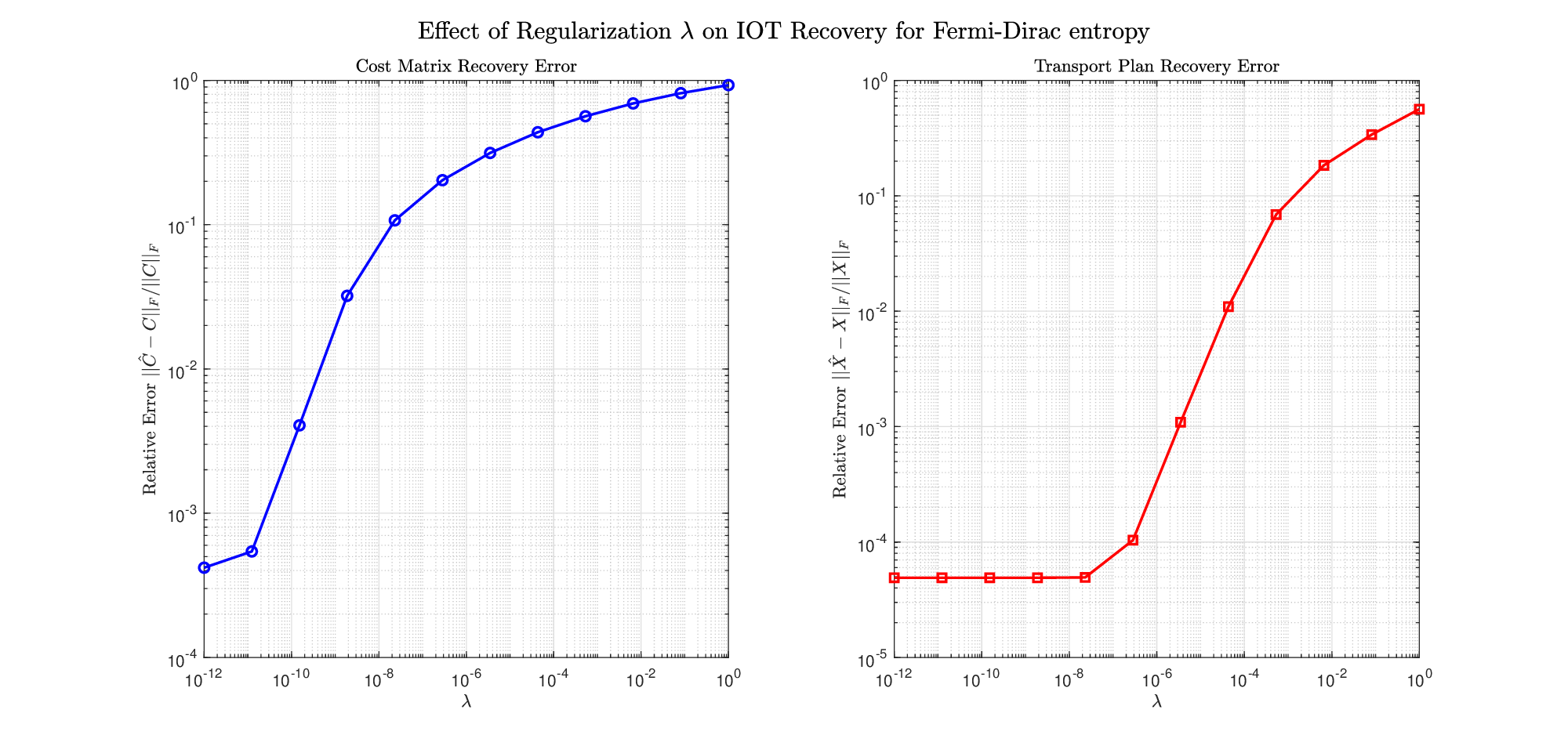}
    \includegraphics[width=0.85\linewidth]{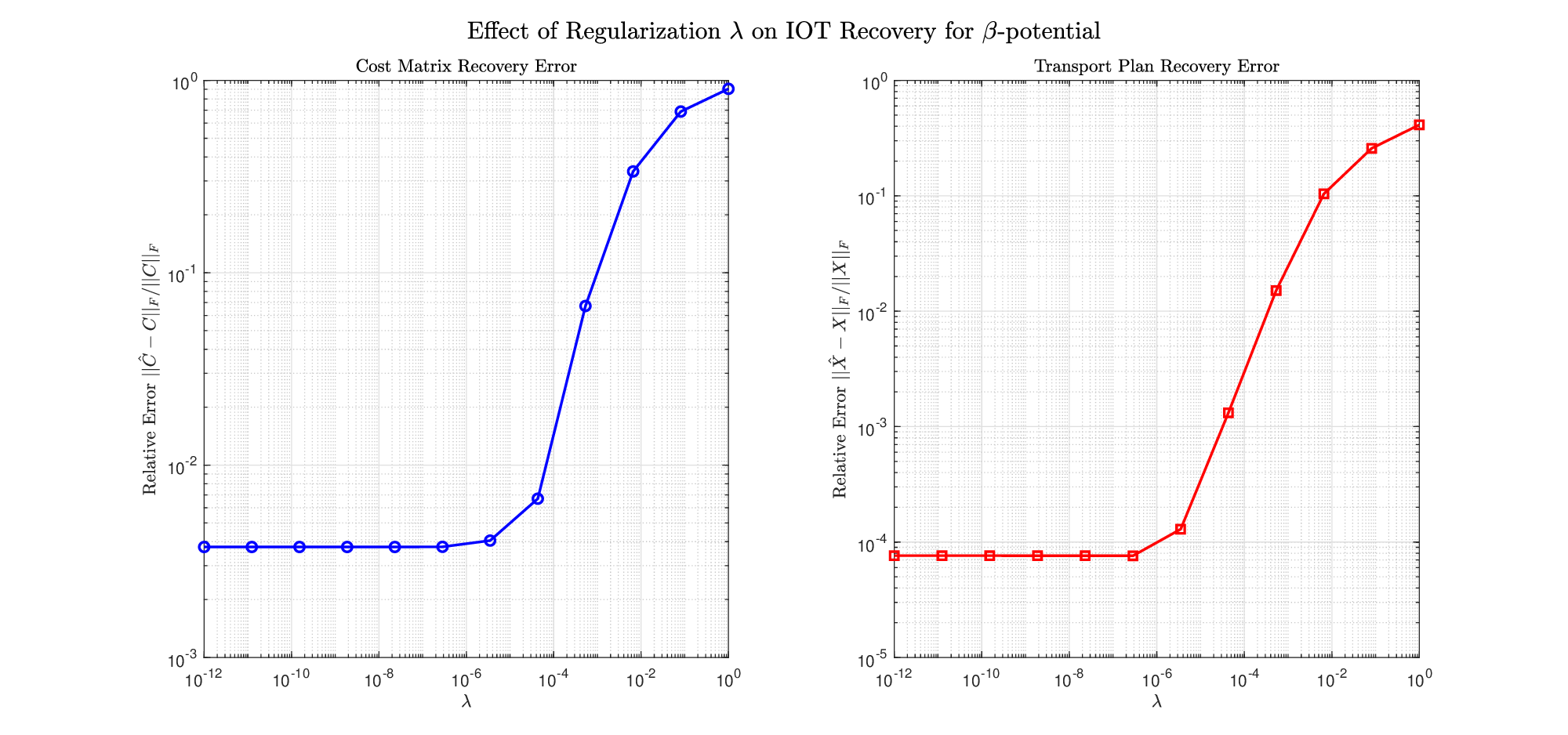}
    \caption{This figure illustrates the impact of the regularization parameter $\lambda$ on the algorithm's performance under the Fermi-Dirac entropy and $\beta$-potential regularizers, where we set $\beta=0.5$ for the latter.}
    \label{fig:lambda2}
\end{figure}

\subsection{Real-World Application - Marriage Matching Preferences}

To evaluate the performance of our proposed  algorithm on a real-world task, we apply it to the problem of inferring marriage matching preferences using the Dutch Household Survey (DHS) dataset\footnote{https://www.dhsdata.nl/site/users/login}. The data preprocessing, feature extraction (11 demographic and personality traits), clustering of individuals into $k_{\text{cluster}}=50$ types for both men and women, and the general experimental setup, including the formulation for the compared ``RiOT (original)" method, are adopted directly from the publicly available codebase\footnote{https://github.com/ruilin-li/Learning-to-Match-via-Inverse-Optimal-Transport} and methodology presented in \cite{li2019learning}.  This ensures a consistent basis for comparison. Briefly, after preprocessing and clustering, the observed marriage data is aggregated into an empirical matching matrix $\hat{X} \in \mathbb{R}^{50 \times 50}$ between types of men and women, with $U_0 \in \mathbb{R}^{11 \times 50}$ and $V_0 \in \mathbb{R}^{11 \times 50}$ representing the feature centroids for male and female types, respectively. The cost matrix is parameterized as $C(A) = -U_0^T A V_0$, where $A \in \mathbb{R}^{11 \times 11}$ is the interaction (affinity) matrix to be learned. Note that in this case, the constraint set $\mathcal{C} = \{C\in\mathbb{R}^{50 \times 50}: C = -U_0^T A V_0, A \in \mathbb{R}^{11 \times 11} \}$ and one can easily verify that $$
\operatorname{Proj}_{\mathcal{C}}(C) = (U_0^\dagger U_0)^TC(V_0^\dagger V_0).
$$
We employ a 5-fold cross-validation scheme. In each fold, the training portion of the observed pairings $\hat{X}_{\text{train}}^{(f)}$ (along with its marginals $\hat{\mu}_{\text{train}}^{(f)}, \hat{\nu}_{\text{train}}^{(f)}$) is used to learn the cost matrix $C_{rec}^{(f)}$ via different algorithms. For our  method, we use the entropy regularizer $\phi(x) = x \log x - x + 1$. We choose $\gamma = 1$ and $\lambda = 0$. The maximum number of \Cref{alg:IBCD} iterations was set to 100. The predictive performance is evaluated by using the learned cost $C_{rec}^{(f)}$ to generate a predicted matching plan $\pi_{\text{pred}}^{(f)}$ for the *test set marginals* (derived from the held-out pairings $\hat{\pi}_{\text{test}}^{(f)}$). This prediction is done using the entropy regularized optimal transport computed via Sinkhorn algorithm \cite{cuturi2013sinkhorn}. We report the average Root Mean Squared Error (RMSE) and Mean Absolute Error (MAE) between $\pi_{\text{pred}}^{(f)}$ and $\hat{X}_{\text{test}}^{(f)}$ over the 5 folds. \Cref{tab:marriage_matching} presents the comparative results of our algorithm against several baselines and the RIOT method\cite{li2019learning}. The baselines include random predictor model (Random), probabilistic matrix factorization (PMF) \cite{mnih2007probabilistic}, classic SVD method (SVD) \cite{koren2009matrix}, and item-based collaborative filtering model (itemKNN) \cite{cremonesi2010performance}. As shown in \Cref{tab:marriage_matching}, our method demonstrates competitive performance and beats other methods in both measures with the fastest speed.

\begin{table}
\centering
\begin{tabular}{lccc}
\hline
Model & RMSE & MAE & Time(s) \\
\hline
Random & 0.00111 & 0.00073 & 0.00640 \\
PMF & 0.00334 & 0.00244 & 0.02018 \\
SVD & 0.00335 & 0.00245 & 0.02090 \\
itemKNN & 0.00096 & 0.00063 & 0.07007 \\
RIOT & 0.00094 & 0.00060 & 0.72104 \\
\textbf{IBCD} & \textbf{0.00091} & \textbf{0.00057} & \textbf{0.00411} \\
\hline
\end{tabular}
\caption{Results of experiments on marriage matching.}
\label{tab:marriage_matching}
\end{table}

\section{Conclusion and Future Work}\label{sec:conclusion}
In this work, we have advanced the understanding of inverse optimal transport (IOT) by developing a comprehensive well-posedness theory for IOT problems regularized by general Bregman functions. We extended the single-level convex reformulation beyond the entropic case and conducted a rigorous analysis of its existence for solutions. From an algorithmic perspective, we proposed an efficient Inexact Block Coordinate Descent (IBCD) method tailored to this framework, proved its linear convergence rate under common quadratic penalties, and demonstrated its practical effectiveness through extensive numerical experiments. These contributions establish a more robust and general foundation for inferring underlying cost structures from observed transport phenomena.

Several interesting directions emerge for future work. First, addressing the challenge of sparse observed transport plans remains critical. When observations $\hat{X}$ are sparse, formulations that rely on properties of dense transport plans, such as Bregman divergences (including the KL divergence), which involve terms like $\log \hat{X}_{ij}$, may become problematic. Future research could explore alternative loss functions or imputation strategies that are robust to sparsity, or develop IOT models that explicitly enforce structured sparsity in the inferred cost or predicted transport plans.

Second, extending the framework to learn a single underlying cost function from multiple observed transport plans is a particularly compelling direction. In this scenario, different transport plans might arise from the same cost structure but under varying marginal distributions or contexts. This setting offers the opportunity to leverage richer information for more robust cost inference. Developing models and algorithms capable of aggregating information from multiple, potentially noisy transport plans would significantly improve the applicability of IOT and unlock new practical and theoretical insights.

\appendix
\section{Proof of Lemma \ref{lem:forward_kkt}}
\label{app:proof_lemma}
\begin{proof}
It suffices to show (b) since both (a) and (c) can be derived by direct computation from (b). Consider the primal Bregman-regularized optimal transport problem:
\begin{equation}
    \min_{X \ge 0} \; \langle C, X \rangle + \gamma \phi(X) \quad \text{subject to} \quad X\mathbf{1} = \mu,\; X^\top\mathbf{1} = \nu.
\end{equation}
Let $u, v \in \mathbb{R}^n$ be Lagrange multipliers associated with the two marginal constraints. The Lagrangian function is defined as:
\begin{align}
    \mathcal{L}(X, u, v) &= \langle C, X \rangle + \gamma\phi(X) - \langle u, X\mathbf{1} - \mu \rangle - \langle v, X^\top\mathbf{1} - \nu \rangle \notag \\
    &= \sum_{i,j=1}^n \left( (C_{ij} - u_i - v_j) X_{ij} + \gamma\phi(X_{ij}) \right) + \langle u, \mu \rangle + \langle v, \nu \rangle.
\end{align}

For fixed $(u,v)$, minimizing $\mathcal{L}(\cdot, u, v)$ over $X \ge 0$ decouples into independent scalar subproblems. For each pair $(i,j)$, define  $z_{ij} := \frac{u_i + v_j - C_{ij}}{\gamma}$ and consider the function $f_{ij}(x) = -\gamma z_{ij} x + \gamma\phi(x)$ for $x \ge 0$. Since $\phi$ is strictly convex by Assumption \ref{asmp:breg}, $f_{ij}$ is strictly convex and attains a unique minimizer $X_{ij}^C$.

The optimality condition for this constrained convex problem is given by the inequality:
\begin{equation}
    f_{ij}'(X_{ij}^C)(x - X_{ij}^C) \ge 0 \quad \forall x \ge 0,
\end{equation}
where $f_{ij}'(x) = \gamma(\phi'(x) - z_{ij})$. Equivalently,
\begin{equation}
    (\phi'(X_{ij}^C) - z_{ij})(x - X_{ij}^C) \ge 0 \quad \forall x \ge 0.
\end{equation}
Since $\phi'$ is monotone and $\psi = \phi^*$, this condition characterizes $X_{ij}^C$ as the projection of $\nabla\psi(z_{ij})$ onto $[0,\infty)$. More explicitly, we have:
\begin{equation}
    X_{ij}^C = \begin{cases}
        \psi^\prime(z_{ij}) & \text{if } \psi^\prime(z_{ij}) \ge 0, \\
        0 & \text{otherwise},
    \end{cases}
\end{equation}
which can be written compactly as $X^C = \bigl(\nabla\psi(z)\bigr)_+$. This yields the first condition in \eqref{eqn:KKT_bregman}.

The primal feasibility conditions $X\mathbf{1} = \mu$ and $X^\top\mathbf{1} = \nu$ constitute the remaining KKT conditions. Conversely, if $(X,u,v)$ satisfies \eqref{eqn:KKT_bregman}, then $X$ minimizes $\mathcal{L}(\cdot,u,v)$ by construction and is primal feasible, hence $(X,u,v)$ is a saddle point of the Lagrangian and therefore optimal.
\end{proof}

\bibliographystyle{plain}
\bibliography{ref}

@article{qi2013semismooth,
  title={A semismooth Newton method for the nearest Euclidean distance matrix problem},
  author={Qi, Hou-Duo},
  journal={SIAM Journal on Matrix analysis and applications},
  volume={34},
  number={1},
  pages={67--93},
  year={2013},
  publisher={SIAM}
}

@article{dessein2018regularized,
  title={Regularized optimal transport and the rot mover's distance},
  author={Dessein, Arnaud and Papadakis, Nicolas and Rouas, Jean-Luc},
  journal={Journal of Machine Learning Research},
  volume={19},
  number={15},
  pages={1--53},
  year={2018}
}

@inproceedings{chiu2022discrete,
  title={Discrete probabilistic inverse optimal transport},
  author={Chiu, Wei-Ting and Wang, Pei and Shafto, Patrick},
  booktitle={International Conference on Machine Learning},
  pages={3925--3946},
  year={2022},
  organization={PMLR}
}

@inproceedings{wang2023self,
  title={Self-supervised video summarization guided by semantic inverse optimal transport},
  author={Wang, Yutong and Xu, Hongteng and Luo, Dixin},
  booktitle={Proceedings of the 31st ACM International Conference on Multimedia},
  pages={6611--6622},
  year={2023}
}

@inproceedings{yu2022explainable,
  title={Explainable legal case matching via inverse optimal transport-based rationale extraction},
  author={Yu, Weijie and Sun, Zhongxiang and Xu, Jun and Dong, Zhenhua and Chen, Xu and Xu, Hongteng and Wen, Ji-Rong},
  booktitle={Proceedings of the 45th international ACM SIGIR conference on research and development in information retrieval},
  pages={657--668},
  year={2022}
}

@article{carlier2023sista,
  title={Sista: learning optimal transport costs under sparsity constraints},
  author={Carlier, Guillaume and Dupuy, Arnaud and Galichon, Alfred and Sun, Yifei},
  journal={Communications on Pure and Applied Mathematics},
  volume={76},
  number={9},
  pages={1659--1677},
  year={2023},
  publisher={Wiley Online Library}
}

@article{persiianov2024inverse,
  title={Inverse Entropic Optimal Transport Solves Semi-supervised Learning via Data Likelihood Maximization},
  author={Persiianov, Mikhail and Asadulaev, Arip and Andreev, Nikita and Starodubcev, Nikita and Baranchuk, Dmitry and Kratsios, Anastasis and Burnaev, Evgeny and Korotin, Alexander},
  journal={arXiv preprint arXiv:2410.02628},
  year={2024}
}

@article{dupuy2019estimating,
  title={Estimating matching affinity matrices under low-rank constraints},
  author={Dupuy, Arnaud and Galichon, Alfred and Sun, Yifei},
  journal={Information and Inference: A Journal of the IMA},
  volume={8},
  number={4},
  pages={677--689},
  year={2019},
  publisher={Oxford University Press}
}

@article{stuart2020inverse,
  title={Inverse optimal transport},
  author={Stuart, Andrew M and Wolfram, Marie-Therese},
  journal={SIAM Journal on Applied Mathematics},
  volume={80},
  number={1},
  pages={599--619},
  year={2020},
  publisher={SIAM}
}

@article{andrade2023sparsistency,
  title={Sparsistency for inverse optimal transport},
  author={Andrade, Francisco and Peyr{\'e}, Gabriel and Poon, Clarice},
  journal={arXiv preprint arXiv:2310.05461},
  year={2023}
}

@article{li2019learning,
  title={Learning to match via inverse optimal transport},
  author={Li, Ruilin and Ye, Xiaojing and Zhou, Haomin and Zha, Hongyuan},
  journal={Journal of machine learning research},
  volume={20},
  number={80},
  pages={1--37},
  year={2019}
}

@article{ma2020learning,
  title={Learning cost functions for optimal transport},
  author={Ma, Shaojun and Sun, Haodong and Ye, Xiaojing and Zha, Hongyuan and Zhou, Haomin},
  journal={arXiv preprint arXiv:2002.09650},
  year={2020}
}

@article{gonzalez2024identifiability,
  title={Identifiability of the Optimal Transport Cost on Finite Spaces},
  author={Gonz{\'a}lez-Sanz, Alberto and Groppe, Michel and Munk, Axel},
  journal={arXiv preprint arXiv:2410.23146},
  year={2024}
}

@article{gonzalez2024nonlinear,
  title={Nonlinear inverse optimal transport: Identifiability of the transport cost from its marginals and optimal values},
  author={Gonz{\'a}lez-Sanz, Alberto and Groppe, Michel and Munk, Axel},
  journal={SIAM Journal on Mathematical Analysis},
  volume={56},
  number={6},
  pages={7808--7829},
  year={2024},
  publisher={SIAM}
}

@article{cuturi2013sinkhorn,
  title={Sinkhorn distances: Lightspeed computation of optimal transport},
  author={Cuturi, Marco},
  journal={Advances in neural information processing systems},
  volume={26},
  year={2013}
}

@book{villani2008optimal,
  title={Optimal transport: old and new},
  author={Villani, C{\'e}dric and others},
  volume={338},
  year={2008},
  publisher={Springer}
}

@article{peyre2019computational,
  title={Computational optimal transport: With applications to data science},
  author={Peyr{\'e}, Gabriel and Cuturi, Marco and others},
  journal={Foundations and Trends{\textregistered} in Machine Learning},
  volume={11},
  number={5-6},
  pages={355--607},
  year={2019},
  publisher={Now Publishers, Inc.}
}

@article{courty2016optimal,
  title={Optimal transport for domain adaptation},
  author={Courty, Nicolas and Flamary, R{\'e}mi and Tuia, Devis and Rakotomamonjy, Alain},
  journal={IEEE transactions on pattern analysis and machine intelligence},
  volume={39},
  number={9},
  pages={1853--1865},
  year={2016},
  publisher={IEEE}
}

@article{kolouri2017optimal,
  title={Optimal mass transport: Signal processing and machine-learning applications},
  author={Kolouri, Soheil and Park, Se Rim and Thorpe, Matthew and Slepcev, Dejan and Rohde, Gustavo K},
  journal={IEEE signal processing magazine},
  volume={34},
  number={4},
  pages={43--59},
  year={2017},
  publisher={IEEE}
}

@article{lorenz2021quadratically,
  title={Quadratically regularized optimal transport},
  author={Lorenz, Dirk A and Manns, Paul and Meyer, Christian},
  journal={Applied Mathematics \& Optimization},
  volume={83},
  number={3},
  pages={1919--1949},
  year={2021},
  publisher={Springer}
}

@article{essid2018quadratically,
  title={Quadratically regularized optimal transport on graphs},
  author={Essid, Montacer and Solomon, Justin},
  journal={SIAM Journal on Scientific Computing},
  volume={40},
  number={4},
  pages={A1961--A1986},
  year={2018},
  publisher={SIAM}
}

@phdthesis{genevay2019entropy,
  title={Entropy-regularized optimal transport for machine learning},
  author={Genevay, Aude},
  year={2019},
  school={Universit{\'e} Paris sciences et lettres}
}

@article{luo1993error,
  title={Error bounds and convergence analysis of feasible descent methods: a general approach},
  author={Luo, Zhi-Quan and Tseng, Paul},
  journal={Annals of Operations Research},
  volume={46},
  number={1},
  pages={157--178},
  year={1993},
  publisher={Springer}
}

@book{rockafellar1997convex,
  title={Convex analysis},
  author={Rockafellar, R Tyrrell},
  volume={28},
  year={1997},
  publisher={Princeton university press}
}

@article{koren2009matrix,
  title={Matrix factorization techniques for recommender systems},
  author={Koren, Yehuda and Bell, Robert and Volinsky, Chris},
  journal={Computer},
  volume={42},
  number={8},
  pages={30--37},
  year={2009},
  publisher={IEEE}
}

@inproceedings{cremonesi2010performance,
  title={Performance of recommender algorithms on top-n recommendation tasks},
  author={Cremonesi, Paolo and Koren, Yehuda and Turrin, Roberto},
  booktitle={Proceedings of the fourth ACM conference on Recommender systems},
  pages={39--46},
  year={2010}
}

@article{mnih2007probabilistic,
  title={Probabilistic matrix factorization},
  author={Mnih, Andriy and Salakhutdinov, Russ R},
  journal={Advances in neural information processing systems},
  volume={20},
  year={2007}
}

@article{liu2019learning,
  title={Learning transport cost from subset correspondence},
  author={Liu, Ruishan and Balsubramani, Akshay and Zou, James},
  journal={arXiv preprint arXiv:1909.13203},
  year={2019}
}

@article{andrade2025learning,
  title={Learning from Samples: Inverse Problems over measures via Sharpened Fenchel-Young Losses},
  author={Andrade, Francisco and Peyr{\'e}, Gabriel and Poon, Clarice},
  journal={arXiv preprint arXiv:2505.07124},
  year={2025}
}

@inproceedings{shi2023understanding,
  title={Understanding and generalizing contrastive learning from the inverse optimal transport perspective},
  author={Shi, Liangliang and Zhang, Gu and Zhen, Haoyu and Fan, Jintao and Yan, Junchi},
  booktitle={International conference on machine learning},
  pages={31408--31421},
  year={2023},
  organization={PMLR}
}

\end{document}